\documentclass[onefignum,onetabnum]{siamart220329}

\usepackage{amsfonts}
\usepackage{amsmath}
\usepackage{amssymb}
\usepackage{mathtools}
\usepackage{graphicx}
\usepackage{epstopdf}
\usepackage[ruled,vlined,algo2e]{algorithm2e}
\usepackage{bbm}
\usepackage{lscape}
\usepackage{yfonts}
\usepackage{verbatim}
\usepackage{xcolor}
\usepackage{hyperref}
\usepackage{subcaption}

\ifpdf
  \DeclareGraphicsExtensions{.eps,.pdf,.png,.jpg}
\else
  \DeclareGraphicsExtensions{.eps}
\fi

\DeclareSymbolFont{lettersA}{U}{txmia}{m}{it}
\DeclareMathSymbol{\R}{\mathord}{lettersA}{"92}
\DeclareMathSymbol{\C}{\mathord}{lettersA}{"83}

\newcommand{\bA}{{\bf A}}
\newcommand{\bb}{{\bf b}}
\newcommand{\bc}{{\bf c}}
\newcommand{\bC}{{\bf C}}

\renewcommand{\sf}{\ensuremath{\{\textswab{s:},\textfrak{f}\}}}

\newcommand{\tol}{\ensuremath{\textsc{tol}}}

\definecolor{ForestGreen}{rgb}{0.18,0.70,0.18}

\allowdisplaybreaks

\newsiamremark{remark}{Remark}
\newsiamremark{hypothesis}{Hypothesis}
\crefname{hypothesis}{Hypothesis}{Hypotheses}
\newsiamthm{claim}{Claim}

\headers{Implicit-Explicit Multirate Infinitesimal Stage-Restart Methods}{A.~C.~Fish, D.~R.~Reynolds, and S.~B.~Roberts}

\title{Implicit-Explicit Multirate Infinitesimal Stage-Restart Methods \thanks{Submitted to the editors DATE.
\funding{The work of Alex Fish and Daniel Reynolds was supported by the U.S. Department of Energy, Office of Science, Office of Advanced Scientific Computing Research, Scientific Discovery through Advanced Computing (SciDAC) Program through the FASTMath Institute, under DOE award DE-SC0021354.\\
The work of Steven Roberts was performed under the auspices of the U.S. Department of Energy by Lawrence Livermore National Laboratory under Contract DEAC52-07NA27344 and was supported by the U.S. Department of Energy, Office of Science, Office of Advanced Scientific Computing Research. LLNL-JRNL-843436.\\
This document was prepared as an account of work sponsored by an agency of the United States government. Neither the United States government nor Lawrence Livermore National Security, LLC, nor any of their employees makes any warranty, expressed or implied, or assumes any legal liability or responsibility for the accuracy, completeness, or usefulness of any information, apparatus, product, or process disclosed, or represents that its use would not infringe privately owned rights. Reference herein to any specific commercial product, process, or service by trade name, trademark, manufacturer, or otherwise does not necessarily constitute or imply its endorsement, recommendation, or favoring by the United States government or Lawrence Livermore National Security, LLC. The views and opinions of authors expressed herein do not necessarily state or reflect those of the United States government or Lawrence Livermore National Security, LLC, and shall not be used for advertising or product endorsement purposes.}}}

\author{
Alex C.~Fish\thanks{Department of Mathematics, Southern Methodist University, Dallas, TX, USA (\email{afish@smu.edu}, \email{reynolds@smu.edu}).}
\and
Daniel R.~Reynolds\footnotemark[2]
\and
Steven B.~Roberts\thanks{Center for Applied Scientific Computing, Lawrence Livermore National Laboratory, Livermore, CA, USA (\email{roberts115@llnl.gov}).}
}

\makeatletter
\newcommand*{\addFileDependency}[1]{
  \typeout{(#1)}
  \@addtofilelist{#1}
  \IfFileExists{#1}{}{\typeout{No file #1.}}
}
\makeatother

\ifpdf
\hypersetup{
  pdftitle={Implicit-Explicit Multirate Infinitesimal Stage-Restart Methods},
  pdfauthor={A.~C. Fish, D.~R. Reynolds, and S.~J.~Roberts}
}
\fi

\begin{document}

\maketitle
\begin{abstract}
Implicit-Explicit (IMEX) methods are flexible numerical time integration methods which solve an initial-value problem (IVP) that is partitioned into stiff and nonstiff processes with the goal of lower computational costs than a purely implicit or explicit approach. A complementary form of flexible IVP solvers are multirate infinitesimal methods for problems partitioned into fast- and slow-changing dynamics, that solve a multirate IVP by evolving a sequence of ``fast'' IVPs using any suitably accurate algorithm. This article introduces a new class of high-order implicit-explicit multirate methods that are designed for multirate IVPs in which the slow-changing dynamics are further partitioned in an IMEX fashion. This new class, which we call implicit-explicit multirate stage-restart (IMEX-MRI-SR), both improves upon the previous implicit-explicit multirate generalized-structure additive Runge Kutta (IMEX-MRI-GARK) methods, and extends multirate exponential Runge Kutta (MERK) methods into the IMEX context. We leverage GARK theory to derive conditions guaranteeing orders of accuracy up to four. We provide second-, third-, and fourth-order accurate example methods and perform numerical simulations demonstrating convergence rates and computational performance in both fixed-step and adaptive-step settings.
\end{abstract}




\section{Introduction}
\label{sec:intro}

Flexible time integration methods for solving systems of initial-value problems (IVPs) have seen growing interest in recent years, largely due to their ability to provide highly accurate approximations of the IVP solution with increased computational efficiency. These integrators strive to reduce computational costs by partitioning the IVP into different components, and then treating each using different step sizes or numerical methods.  Some of the primary families of flexible methods include implicit-explicit (IMEX) partitioning \cite{calvo2001linearly,Kennedy2003ark,Kennedy2019ark,Sandu2015}, linear-nonlinear partitioning \cite{hochbruckexprk,luan_new_2020,luan_multirate_2021,Luan2014,Luan2016}, and multirate  partitioning \cite{gunther_multirate_2016,sandu_class_2019,sarshar_design_2019,Wensch2009}.

IMEX time integration methods solve IVPs in which the right-hand side function $f(t,y(t))$ is additively split into stiff $\{I\}$ and nonstiff $\{E\}$ processes,
\begin{equation}
    \label{eq:imexproblem}
    \begin{split}
    y'(t) &= f(t,y) := f^{\{I\}}(t,y)+f^{\{E\}}(t,y), \quad t \ge t_0,\\
    y(t_0) &= y_0.
    \end{split}
\end{equation}
IMEX methods then couple two different numerical methods to treat these components: $f^{\{I\}}$ typically uses a stiff but computationally expensive solver, whereas $f^{\{E\}}$ may use a cheaper but nonstiff solver.  For example, additive Runge--Kutta (ARK) methods typically combine an A-stable diagonally-implicit Runge--Kutta (DIRK) method with an explicit Runge--Kutta (RK) method.

Similarly, multirate methods solve IVPs in which the right-hand side is additively partitioned into rapidly and slowly evolving dynamics, $\{F\}$ and  $\{S\}$,
\begin{equation}
    \label{eq:multirateproblem}
    \begin{split}
    y'(t) &= f(t,y) := f^{\{F\}}(t,y)+f^{\{S\}}(t,y), \quad t \ge t_0,\\
    y(t_0) &= y_0.
    \end{split}
\end{equation}
Multirate methods then apply numerical methods with different step sizes for each component to save on computation time while retaining a desired level of accuracy.

In this work we combine the above approaches to consider a three-way additively partitioned IVP, wherein the slow partition $f^{\{S\}}$ from \eqref{eq:multirateproblem} is split in an IMEX fashion,
\begin{equation}
    \label{eq:imexmriproblem}
    \begin{split}
    y'(t) &= f^{\{F\}}(t,y) + f^{\{I\}}(t,y) + f^{\{E\}}(t,y),\quad t \ge t_0,\\
    y(t_0) &= y_0.
    \end{split}
\end{equation}
In particular, we add a new class of methods to the ever-growing family of multirate infinitesimal (MRI) methods.  These approximate the solution to \eqref{eq:multirateproblem} or \eqref{eq:imexmriproblem} through solving a sequence of ``fast'' IVPs,
\begin{equation}
    \label{eq:infinitesimalmethod}
    \begin{split}
    v'_i(\theta) &= f^{\{F\}}(\theta,v_i) + g_i(\theta),\quad \theta\in[\theta_{0,i},\theta_{f,i}],\\ v(\theta_{0,i}) &= v_{0,i}.
    \end{split}
\end{equation}
The forcing functions $g_i(\theta)$ incorporate information from the slow dynamics defined by $f^{\{I\}}$ and $f^{\{E\}}$ in a manner defined by the method.  MRI methods assume that these fast IVPs \eqref{eq:infinitesimalmethod} are solved exactly, but in practice these are approximated using an additional ``inner'' numerical method with a smaller step size than the multirate method.  This inner method can itself further decompose the problem through IMEX, linear-nonlinear, or multirate approaches.

Our proposed class of methods is called \emph{Implicit-Explicit Multirate Infinitesimal Stage-Restart} (IMEX-MRI-SR) methods.  Each stage of an IMEX-MRI-SR method consists of evolving a fast IVP followed by an implicit solve. This allows derivation of IMEX-MRI-SR methods by extending a base ARK method. We discuss the role of base methods further in Section \ref{sec:baseconsistency}, particularly focusing on their role in satisfying order conditions.

IMEX-MRI-SR methods are defined by $n_\Omega$ coefficient matrices $\Omega^{\{k\}}\in \mathbb{R}^{s^{\{S\}}\times s^{\{S\}}}$, $k=0,...,n_\Omega-1$, a coefficient matrix $\Gamma\in\mathbb{R}^{s^{\{S\}}\times s^{\{S\}}}$, and an abscissae vector $c^{\{S\}}\in\mathbb{R}^{s^{\{S\}}}$.  Embedded versions of these methods include additional coefficient vectors $\hat{\omega}^{\{k\}} \in \mathbb{R}^{s^{\{S\}}}$ and $\hat{\gamma} \in \mathbb{R}^{s^{\{S\}}}$.  The algorithm for evolving a solution to an IVP of the form \eqref{eq:imexmriproblem} is defined as follows.

{\definition[IMEX-MRI-SR methods for additively partitioned systems]{An IMEX-MRI-SR method evolves the solution to the problem \eqref{eq:imexmriproblem} from $t_n$ to $t_n+H$ according to the following algorithm.}}
\begin{subequations}
\label{eq:imexmrigarksr}
\begin{align}
    &\text{Let: }Y_1^{\{S\}}:=y_n \\
    \notag
    &\text{For: }i=2,...,s^{\{S\}} \\
    \label{eq:fastivp}
    &\begin{cases}
    \text{Let: }& v_i(0):=y_n, \\
    \text{Solve: }& v'_i(\theta)=f^{\{F\}}(t_n+\theta,v_i(\theta))+g_i(\theta),\ \text{for }\theta\in[0,c^{\{S\}}_iH] \\
    & \text{where }g_i(\theta)=\dfrac{1}{c^{\{S\}}_i}\displaystyle\sum_{j=1}^{i-1}\omega_{i,j}\left(\frac{\theta}{c^{\{S\}}_iH}\right)\left(f^{\{E\}}_j+f^{\{I\}}_j\right) \\
    \text{Solve: }&Y^{\{S\}}_i=v_i\left(c^{\{S\}}_iH\right)+H\displaystyle\sum_{j=1}^i\gamma_{i,j}f^{\{I\}}_j,
    \end{cases} \\
    &\text{Let: }y_{n+1}:=Y^{\{S\}}_{s^{\{S\}}},\\
    \label{eq:embedding}
    &\begin{cases}
    \text{Let: }& \hat{v}(0):=y_n, \\
    \text{Solve: }& \hat{v}(\theta)=f^{\{F\}}(t_n+\theta,\hat{v}(\theta))+\hat{g}(\theta),\ \text{for }\theta\in[0,H] \\
    & \text{where } \hat{g}(\theta)=\displaystyle\sum_{j=1}^{s^{\{S\}}-1}\hat{\omega}_{j}\left(\frac{\theta}{H}\right)\left(f^{\{E\}}_j+f^{\{I\}}_j\right) \\
    \text{Solve: }&\hat{y}_{n+1}=\hat{v}\left(H\right)+H\displaystyle\sum_{j=1}^{s^{\{S\}}-1}\hat{\gamma}_{j}f^{\{I\}}_j + H \hat{\gamma}_{s^{\{S\}}}f^{\{I\}}(t_n+H,\hat{y}_{n+1}),
    \end{cases}
\end{align}
\end{subequations}
where $f^{\{E\}}_j := f^{\{E\}}\left(t_n+c^{\{S\}}_jH,Y_j^{\{S\}}\right)$ and $f^{\{I\}}_j := f^{\{I\}}\left(t_n+c^{\{S\}}_jH,Y_j^{\{S\}}\right)$.  Here $y_{n+1}$ is the time-evolved approximation to $y(t_n+H)$, and $\hat{y}_{n+1}$ is an embedded solution used for temporal error esimation.  If temporal error estimation is not needed, then step \eqref{eq:embedding} may be omitted.

{\definition[Slow tendency coefficients]{The coefficients $\omega_{i,j}$ in \eqref{eq:fastivp} are defined as in \cite{chinomona_implicit-explicit_2021}
\begin{equation}
    \label{eq:imexmrisr_coefficients}
    \omega_{i,j}(\tau):=\sum_{k=0}^{n_\Omega-1}\omega^{\{k\}}_{i,j}\tau^k,\quad \overline{\omega}_{i,j}:=\int_0^1\omega_{i,j}(\tau)\mathrm d\tau=\sum_{k=0}^{n_\Omega-1}\frac{\omega^{\{k\}}_{i,j}}{k+1},
\end{equation}
and we refer to $\Omega^{\{k\}}$, $\overline{\Omega}$, and $\Gamma$ as the $s^{\{S\}}\times s^{\{S\}}$ matrices containing the coefficients $\{\omega^{\{k\}}_{i,j}\}$, $\{\overline{\omega}_{i,j}\}$, and $\{\gamma_{i,j}\}$, respectively.  As in \cite{chinomona_implicit-explicit_2021} we assume the first row of the coefficient matrices are identically zero, and that $\Omega^{\{k\}}$ and $\overline{\Omega}$ are strictly lower-triangular.  The embedding functions $\hat{\omega}_j(\tau)$ are defined similarly to \eqref{eq:imexmrisr_coefficients}, with vectors of coefficients $\hat{\omega}^{\{k\}}$.
}}

The rest of this paper is structured as follows.  In the next subsection, we present related methods to IMEX-MRI-SR, discussing both their similarities and their limitations that are improved upon by the proposed methods.  In Section \ref{sec:orderconditions} we prove order conditions for IMEX-MRI-SR methods up to order four, and in Section \ref{sec:linearstability} we examine their linear stability.  In Section \ref{sec:examplemethods} we provide embedded IMEX-MRI-SR methods of orders 2 through 4.  In Section \ref{sec:merkasimexmrisr} we show that a previous class of multirate methods may be reformulated as IMEX-MRI-SR methods, and we use our convergence theory to show previously unproven features of those methods.  We then present numerical results in Section \ref{sec:numericalresults} to validate our convergence theory, and to compare the efficiency of IMEX-MRI-SR methods against existing methods for problems of the form \eqref{eq:imexmriproblem}.  Finally, in Section \ref{sec:conclusions} we conclude this article with a summary of our contributions, and an outlook toward future work.

\subsection{Related methods}
\label{sec:related_methods}

To our knowledge, there exist three MRI algorithms that allow both IMEX partitioning of the slow dynamics and infinitesimal treatment of the fast dynamics.  These are the first-order accurate ``Lie--Trotter'' \cite{Estep2008,Ropp2005} and the second-order accurate ``Strang--Marchuk'' \cite{Marchuk1968,Strang1968} operator splitting methods, and the recent fourth-order IMEX-MRI-GARK \cite{chinomona_implicit-explicit_2021} method, an IMEX variation of MRI-GARK \cite{sandu_class_2019} methods.  Both Lie--Trotter and Strang--Marchuk operate by sequentially applying distinct solvers to each component, only communicating with one another through the initial conditions applied within each sub-solve.  Variations of both classes of methods for the IMEX multirate splitting \eqref{eq:imexmriproblem} are shown in \cite{chinomona_implicit-explicit_2021}.  However, due to their weak ``initial condition'' coupling, Lie--Trotter and Strang--Marchuk are limited to at most first and second order accuracy in time, regardless of the order of accuracy of each component solver.

IMEX-MRI-GARK methods are organized similarly to IMEX-MRI-SR, in that they advance the solution by alternating between evolving fast IVPs and solving implicit algebraic equations involving $f^{\{I\}}$, and they use the result from each stage to provide a contribution to $g_i(\theta)$ for later stages.  However, unlike IMEX-MRI-SR methods, IMEX-MRI-GARK methods evolve each fast IVP over an interval $[c_{i-1}H, c_iH]$, with an initial condition given as the result of the previous stage.  In each stage, either a fast evolution or implicit solve may occur, with this choice dictated by the abscissae: when $\Delta c_i := c_i-c_{i-1} > 0$ a fast evolution occurs, but when $\Delta c_i = 0$ an algebraic system must be solved.  In all existing implicit MRI-GARK and IMEX-MRI-GARK methods, authors have derived schemes by beginning with a base DIRK or ARK method, and then introduced additional internal stages to ensure an alternating pattern of $\Delta c_i\ne 0$ followed by $\Delta c_{i+1} = 0$, thereby ensuring an appropriate structure  \cite{chinomona_implicit-explicit_2021,sandu_class_2019}.  However, these methods inherently require abscissae vectors $c^{\{S\}}$ that are non-decreasing.  This is generally an uncommon feature in RK methods, especially for orders of accuracy higher than two \cite{calvo2001linearly,Kennedy2003ark,Kennedy2019dirk,Kennedy2019ark,roberts2021implicit,roberts2020coupled,Sandu2015}, so there are relatively few base DIRK and ARK methods available for deriving new MRI-GARK and IMEX-MRI-GARK methods.  Additionally, the ``padding'' process for adding internal stages to ensure $\Delta c_i=0$ is not obvious, and frequently results in an overly complicated trial-and-error process to decide where to insert stages.  As a result of these challenges, no authors have successfully created IMEX-MRI-GARK methods with embeddings for temporal error estimation.

A second class of methods that are closely related to IMEX-MRI-SR are multirate exponential Runge--Kutta (MERK) methods \cite{luan_new_2020}.  While these do not support implicitness at the slow time scale (i.e., $f^{\{I\}}=0$), and they assume that the fast partition is linear (i.e., $f^{\{F\}}(t,y) = \mathcal{L}y$), their structure matches \eqref{eq:imexmrigarksr}.  Each internal stage is computed through evolving a fast IVP over an interval $[0,c_iH]$, using a forcing function that is determined through the values of $f^{\{E\}}$ at previous slow stages.

\section{IMEX-MRI-SR Order Conditions}
\label{sec:orderconditions}

Similar to IMEX-MRI-GARK methods, we derive order conditions for IMEX-MRI-SR methods by first representing the algorithm in GARK form.  Due to the 3-component partitioning \eqref{eq:imexmriproblem}, we must identify GARK coefficients $\bA^{\{\sigma,\nu\}}$, $\bb^{\{\nu\}}$, and $\bc^{\{\nu\}}$ for $\sigma\in\{S,F\}$ and $\nu\in\{I,E,F\}$.  We refer to $\bA^{\{F,\nu\}}$, $\nu\in\{I,E,F\}$ as the fast GARK tables, and to $\bA^{\{S,\nu\}}$, $\nu\in\{I,E,F\}$ as the slow GARK tables. We assume the fast IVP \eqref{eq:fastivp} is solved with one step of a sufficiently accurate Runge--Kutta method having $s^{\{F\}}$ stages, and defined by the Butcher table $(A^{\{F\}}$, $b^{\{F\}}$, $c^{\{F\}})$.  The $j^{th}$ sub-stage in computing the solution to the fast IVP $v_i'(\theta)$ is given by
\begin{align}
\begin{split}
    \label{eq:faststages}
    V_{i,j} 
    = y_n &+ c^{\{S\}}_iH\sum_{k=1}^{s^{\{F\}}} a^{\{F\}}_{j,k}f^{\{F\}}_{i,k} +  H\sum_{k=1}^{s^{\{F\}}}a^{\{F\}}_{j,k}\sum_{\ell=1}^{i-1} \omega_{i,\ell}(c^{\{F\}}_k)\left(f^{\{E\}}_\ell+f^{\{I\}}_\ell\right)
\end{split}
\end{align}
for $i=1,...,s^{\{S\}}$, and $j=1,...,s^{\{F\}}$, where $f^{\{F\}}_{i,k}:=f^{\{F\}}(t_n+c^{\{F\}}_kH,V_{i,k})$, $f^{\{E\}}_\ell := f^{\{E\}}(t_n+c_{\ell}^{\{S\}}H,Y_\ell^{\{S\}})$ and $f^{\{I\}}_\ell := f^{\{I\}}(t_n+c_{\ell}^{\{S\}}H,Y_\ell^{\{S\}})$. The IMEX-MRI-SR method's slow stages $Y_i^{\{S\}}$, $i=1,\ldots,s^{\{S\}}$, are then computed by
\begin{align}
\begin{split}
    \label{eq:slowstages}
    Y_i^{\{S\}} 
    = y_n &+ c^{\{S\}}_iH\sum_{j=1}^{s^{\{F\}}} b^{\{F\}}_{j}f^{\{F\}}_{i,j}\\
    &+H\sum_{j=1}^{s^{\{F\}}}b^{\{F\}}_{j}\sum_{\ell=1}^{i-1} \omega_{i,\ell}\left(c^{\{F\}}_j\right)\left(f^{\{E\}}_\ell+f^{\{I\}}_\ell\right) + H\sum_{\ell=1}^{i}\gamma_{i,\ell} f^{\{I\}}_\ell.
\end{split}
\end{align}

We leverage GARK theory \cite{Sandu2015} to construct order conditions as in \cite{chinomona_implicit-explicit_2021,sandu_class_2019}. Since the slow partitions share the same stages, $Y_i^{\{S\}}$, these methods have six GARK matrices, $\bA^{\{F,F\}}$, $\bA^{\{F,E\}}$, $\bA^{\{F,I\}}$, $\bA^{\{S,F\}}$, $\bA^{\{S,E\}}$, and $\bA^{\{S,I\}}$ and three GARK vectors $\bb^{\{F\}}$, $\bb^{\{E\}}$, and $\bb^{\{I\}}$,
\[
  \begin{array}{l|l|l}
  \bA^{\{F,F\}} & \bA^{\{F,E\}} & \bA^{\{F,I\}}  \\\hline
  \bA^{\{S,F\}} & \bA^{\{S,E\}} & \bA^{\{S,I\}}
  \\\hline
  \bb^{\{F\},T} & \bb^{\{E\},T} & \bb^{\{I\},T}
  \end{array}.
\]
Here, $\bA^{\{F,F\}}$ is a square $s^{\{SF\}}\times s^{\{SF\}}$ matrix (where we define $s^{\{SF\}}=s^{\{S\}}\cdot s^{\{F\}}$), containing the coefficients relating the fast stages $\{V_{i,j}\}$ to each other. It is block-diagonal because $V_{i,j}$ depends only on $V_{i,k}$, $k=1,...,s^{\{F\}}$ through the $A^{\{F\}}$ coefficients, and never on $V_{\ell j}$, $\ell\ne i$. These $s^{\{F\}}\times s^{\{F\}}$ block-diagonal elements are named $A^{\{F,F,i\}}$.

$\bA^{\{F,E\}}$ and $\bA^{\{F,I\}}$ are tall $s^{\{SF\}}\times s^{\{S\}}$ matrices relating the fast stages $\{V_{i,j}\}$ to the explicit and implicit function evaluations of the slow stages $\{Y_i^{\{S\}}\}$, comprised of $s^{\{S\}}$ blocks named $A^{\{F,E,i\}}$.

$\bA^{\{S,F\}}$ is a wide $s^{\{S\}}\times s^{\{SF\}}$ matrix, containing the coefficients relating the slow stages $\{Y_i^{\{S\}}\}$ to the fast stages $\{V_{i,j}\}$. We name these $s^{\{S\}}$ total $s^{\{S\}}\times s^{\{F\}}$ blocks that comprise $\bA^{\{S,F\}}$ as $A^{\{S,F,i\}}$. Each $A^{\{S,F,i\}}$ contains at most one row of non-zero entries, located in row $i$, because $Y_i^{\{S\}}$ depends only on $V_{i,j}$, $j=1,...,s^{\{F\}}$, and never on $V_{\ell,j}$, $\ell \ne i$.

$\bA^{\{S,E\}}$ and $\bA^{\{S,I\}}$ are square $s^{\{S\}}\times s^{\{S\}}$ matrices relating the slow stages $\{Y_i^{\{S\}}\}$ to the explicit and implicit function evaluations of those slow stages.

The vectors $\bb^{\{\sigma\}}$ equal the last rows of $\bA^{\{S,\sigma\}}$, $\sigma\in\{F,E,I\}$, because IMEX-MRI-SR methods have the first-same-as-last (FSAL) property, where the last stage is used as the solution to the step.

When an embedding \eqref{eq:embedding} is included, it will correspond to three additional GARK vectors, $\hat{\bb}^{\{F\}}$  $\hat{\bb}^{\{E\}}$ and $\hat{\bb}^{\{I\}}$.  Due to the structural similarity of \eqref{eq:embedding} to the last stage $Y_{s^{\{S\}}}^{\{S\}}$, the contents of these vectors will only differ from $\bb^{\{F\}}$, $\bb^{\{E\}}$ and $\bb^{\{I\}}$ through their dependence on $\hat{\omega}_{j}^{\{k\}}$ and $\hat{\gamma}_{j}$ instead of $\omega_{s^{\{S\}},j}^{\{k\}}$ and $\gamma_{s^{\{S\}},j}$.  Thus the order conditions that follow for the primary GARK matrices and vectors can be applied to the embedding as well.

With the above simplifications, the GARK tableau can be expressed in block-matrix form as
\[
  \begin{array}{lll|l|l}
  A^{\{F,F,1\}} & & & A^{\{F,E,1\}} & A^{\{F,I,1\}}  \\
  & \ddots & & \vdots & \vdots  \\
  & & A^{\{F,F,s^{\{S\}}\}} & A^{\{F,E,s^{\{S\}}\}} & A^{\{F,I,s^{\{S\}}\}} \\\hline
  A^{\{S,F,1\}} & \cdots & A^{\{S,F,s^{\{S\}}\}} & \bA^{\{S,E\}} & \bA^{\{S,I\}}
  \\\hline
  & \bb^{\{F\},T} & & \bb^{\{E\},T} & \bb^{\{I\},T}
  \end{array}.
\]
A GARK method with this tabular structure has stage update formulas
\begin{subequations}
\begin{equation}
    \label{eq:GARKfaststages}
    V_{i,j} = y_n + H\sum_{k=1}^{s^{\{F\}}} a^{\{F,F,i\}}_{j,k}f^{\{F\}}_{i,j} +H\sum_{k=1}^{s^{\{F\}}}a^{\{F,E,i\}}_{j,k}f^{\{E\}}_j+H\sum_{k=1}^{s^{\{F\}}}a^{\{F,I,i\}}_{j,k}f^{\{I\}}_j,
\end{equation}
\begin{equation}
    \label{eq:GARKslowstages}
    Y_i^{\{S\}} = y_n + H\sum_{j=1}^{s^{\{F\}}} a^{\{S,F,i\}}_{i,j}f^{\{F\}}_{ij}+H\sum_{j=1}^{s^{\{F\}}}\bA^{\{S,E\}}_{i,j}f^{\{E\}}_j+H\sum_{j=1}^{s^{\{F\}}}\bA^{\{S,I\}}_{i,j}f^{\{I\}}_j.
\end{equation}
\end{subequations}
By matching coefficients in \eqref{eq:faststages} and \eqref{eq:GARKfaststages}, we identify the fast GARK coefficients,
\begin{subequations}
\label{eq:AFXindices}
\begin{align}
    \label{eq:AFFindices}
    a^{\{F,F,i\}}_{j,k} &= c^{\{S\}}_ia^{\{F\}}_{j,k},\\
    \label{eq:AFEindices}
    a^{\{F,E,i\}}_{j,\ell} &= \sum_{l=1}^{s^{\{F\}}}a^{\{F\}}_{j,k}\omega_{i,\ell}(c^{\{F\}}_l)
    = \sum_{l=1}^{s^{\{F\}}}\sum_{k=0}^{n_\Omega-1}\omega^{\{k\}}_{i,\ell}a^{\{F\}}_{j,l}c^{\{F\}\times k}_l,\\
    \label{eq:AFIindices}
    a^{\{F,I,i\}}_{j,\ell} &= a^{\{F,E,i\}}_{j,\ell} =  \sum_{k=1}^{s^{\{F\}}}b^{\{F\}}_{j,k}\omega_{i,\ell}(c^{\{F\}}_k),
\end{align}
\end{subequations}
where the superscript $\times k$ denotes element-wise exponentiation of a vector by $k$.  Converting these to matrix form, we have the fast GARK tables,
\begin{subequations}
\label{eq:AFXtables}
\begin{align}
    \label{eq:AFFtable}
    \bA^{\{F,F\}} &= C^{\{S\}}\otimes A^{\{F\}} \in \mathbb{R}^{s^{\{SF\}}\times s^{\{SF\}}},\\
    \label{eq:AFEtable}
    \bA^{\{F,E\}} &= \sum_{k=0}^{n_\Omega-1}\Omega^{\{k\}}\otimes A^{\{F\}}c^{\{F\}\times k} \in \mathbb{R}^{s^{\{SF\}}\times s^{\{S\}}},\\
    \label{eq:AFItable}
    \bA^{\{F,I\}} &= \bA^{\{F,E\}}
    = \sum_{k=0}^{n_\Omega-1}\Omega^{\{k\}}\otimes A^{\{F\}}c^{\{F\}\times k} \in \mathbb{R}^{s^{\{SF\}}\times s^{\{S\}}},
\end{align}
\end{subequations}
where $\otimes$ denotes the Kronecker product and $C^{\{\sigma\}}=\operatorname{diag}(c^{\{\sigma\}})$.

We similarly find the slow GARK table coefficients by comparing \eqref{eq:slowstages} and \eqref{eq:GARKslowstages},
\begin{subequations}
\label{eq:ASXindices}
\begin{align}
    \label{eq:ASFindices}
    a^{\{S,F,i\}}_{i,j} &= c^{\{S\}}_ib^{\{F\}}_j,\\
    \label{eq:ASEindices}
    a^{\{S,E\}}_{i,\ell} &= \sum_{k=1}^{s^{\{F\}}}b^{\{F\}}_{k}\omega_{i,\ell}(c^{\{F\}}_k)
    = \sum_{k=0}^{n_\Omega-1}\omega^{\{k\}}_{i,\ell}b^{\{F\},T}c^{\{F\}\times k},\\
    \label{eq:ASIindices}
    a^{\{S,I\}}_{i,\ell} &= \sum_{k=1}^{s^{\{F\}}}b^{\{F\}}_{k}\omega_{i,\ell}\left(c^{\{F\}}_j\right) + \gamma_{i,\ell}
    = \sum_{k=0}^{n_\Omega-1}\omega^{\{k\}}_{i,\ell}b^{\{F\},T}c^{\{F\}\times k} + \gamma_{i,\ell}.
\end{align}
\end{subequations}
Due to the infinitesimal nature of the fast method, we assume that it satisfies all bushy-tree order conditions,
\[
  b^{\{F\},T}c^{\{F\}\times k}=\frac{1}{k+1}, \quad k=0,...,n_\Omega-1.
\]
Leveraging this, and examining \eqref{eq:ASXindices}, the slow GARK tables in matrix-form are
\begin{subequations}
\label{eq:ASXtables}
\begin{align}
    \label{eq:ASFtable}
    \bA^{\{S,F\}} &= C^{\{S\}}\otimes b^{\{F\},T} \in \mathbb{R}^{s^{\{S\}}\times s^{\{SF\}}},\\
    \label{eq:ASEtable}
    \bA^{\{S,E\}} &= \sum_{k=0}^{n_\Omega-1} \Omega^{\{k\}}\frac{1}{k+1}
    = \overline{\Omega} \in \mathbb{R}^{s^{\{S\}}\times s^{\{S\}}},\\
    \label{eq:ASItable}
    \bA^{\{S,I\}} &= \sum_{k=0}^{n_\Omega-1} \Omega^{\{k\}}\frac{1}{k+1} + \Gamma
    = \overline{\Omega} + \Gamma \in \mathbb{R}^{s^{\{S\}}\times s^{\{S\}}}.
\end{align}
\end{subequations}

We additionally define the following variables, knowing that an IMEX-MRI-SR method has the FSAL property with respect to the slow stages $Y^{\{S\}}$,
\begin{subequations}
\label{eq:btables}
\begin{align}
    \label{eq:bFtable}
    \bb^{\{F\},T} &= e_{s^{\{S\}}}^T\bA^{\{S,F\}}
    = (C^{\{S\}}e_{s^{\{S\}}}\otimes b^{\{F\},T})
    = (e_{s^{\{S\}}}\otimes b^{\{F\}})^T,\\
    \label{eq:bEtable}
    \bb^{\{E\},T}&= e_{s^{\{S\}}}^T\bA^{\{S,E\}}
    =  e_{s^{\{S\}}}^T\overline{\Omega},\\
    \label{eq:bItable}
    \bb^{\{I\},T}&= e_{s^{\{S\}}}^T\bA^{\{S,I\}}
    =  e_{s^{\{S\}}}^T(\overline{\Omega} + \Gamma),
\end{align}
\end{subequations}
where $e_{s^{\{S\}}}$ is an $s^{\{S\}}$-length vector of all zeroes except a one in the last position.

\subsection{Base Consistency}
\label{sec:baseconsistency}

If $f^{\{F\}}(t,y)=0$, then an IMEX-MRI-SR method reduces to a simple ARK method defined by slow explicit and implicit base methods,
\begin{align*}
    (A^{\{E\}},b^{\{E\}},c^{\{E\}})&=(\bA^{\{S,E\}},\bb^{\{E\}},\bA^{\{S,E\}}\mathbbm{1}^{s^{\{S\}}}),\\
    (A^{\{I\}},b^{\{I\}},c^{\{I\}})&=(\bA^{\{S,I\}},\bb^{\{I\}},\bA^{\{S,I\}}\mathbbm{1}^{s^{\{S\}}}),
\end{align*}
where $\mathbbm{1}^{s^{\{S\}}}$ is a vector of ones with length $s^{\{S\}}$.  This ARK method has stages
\begin{align*}
  Y_i^{\{S\}} = y_n &+ H\sum_{j=1}^{i-1}a^{\{E\}}_{i,j}f^{\{E\}}(t_n+c^{\{E\}}_jH,Y_j^{\{S\}})\\ &+ H\sum_{j=1}^{i}a^{\{I\}}_{i,j}f^{\{I\}}(t_n+c^{\{I\}}_jH,Y_j^{\{S\}}),
\end{align*}
with the last stage being equal to the solution for the step. We will refer to this ARK method as the slow base method throughout the derivation of the order conditions.

{\remark[First order conditions]{As long as the slow base method and the arbitrary fast method are order one or higher, there are no additional first-order coupling conditions for a GARK method. Thus, the effective first-order IMEX-MRI-SR condition is that $\bA^{\{S,E\}}=\overline{\Omega}$ and $\bA^{\{S,I\}}=\overline{\Omega}+\Gamma$ form an order one ARK method, which can be simply achieved if both $\bA^{\{S,E\}}$ and $\bA^{\{S,I\}}$ have first-order accuracy.}}

{\remark[Deriving IMEX-MRI-SR methods from existing ARK methods]{Due to the IMEX-MRI-SR structure, a base ARK method should be stiffly-accurate; otherwise it must first be converted to stiffly-accurate form by appending the $b$ vectors to the bottom and pad a column of zeros to the right of the ARK's $A$ matrices. \label{remark:existingarkbase} }}

\subsection{Kronecker Product Identities}
\label{sec:kroneckeridentities}
In the ensuing derivations we leverage the following identities:
\begin{subequations}
\begin{align}
    \label{eq:krontranspose}
    (A\otimes B)^T &= A^T\otimes B^T,\\
    \label{eq:kronproduct}
    (A\otimes B)(C \otimes D) &= (AC)\otimes (BD),\\
    \label{eq:kronrowsum}
    (A\otimes v)\mathbbm{1}^{\{r(A)\}} &= (A\mathbbm{1}^{\{r(A)\}})\otimes v,\\
    \label{eq:kronrowsum2}
    (v\otimes A)\mathbbm{1}^{\{r(A)\}} &= v\otimes (A\mathbbm{1}^{\{r(A)\}}),
\end{align}
\end{subequations}
where $A,B,C,D$ are arbitrary matrices with compatible dimemsions, $v$ is an arbitrary vector, and $\mathbbm{1}^{\{r(A)\}}$ is a vector of ones with length equal to the number of rows of $A$. Identities \eqref{eq:krontranspose} and \eqref{eq:kronproduct} can be found in \cite{Langville2004}. Identities \eqref{eq:kronrowsum} and \eqref{eq:kronrowsum2} can be shown through elementary computation.

\subsection{GARK Internal Consistency}
\label{sec:internalconsistency}
\begin{theorem}[GARK Internal Consistency]
An IMEX-MRI-SR method satisfies the GARK internal consistency conditions,
\begin{subequations}
\begin{align}
    \bc^{\{F,F\}} &= \bc^{\{F,E\}}=\bc^{\{F,I\}},\\
    \bc^{\{S,F\}} &= \bc^{\{S,E\}}=\bc^{\{S,I\}},
\end{align}
\end{subequations}
where $c^{\{F,\nu\}}=A^{\{F,\nu\}}\mathbbm{1}^{s^{\{SF\}}}$, $c^{\{S,\nu\}}=\bA^{\{S,\nu\}}\mathbbm{1}^{s^{\{S\}}}$, $\nu\in\{I,E,F\}$, if the following conditions hold:
\begin{subequations}
\begin{align}
    \Omega^{\{0\}}\mathbbm{1}^{s^{\{S\}}} &= c^{\{S\}}, \\
    \Omega^{\{k\}}\mathbbm{1}^{s^{\{S\}}} &= 0^{s^{\{S\}}},\ k=1,...,n_\Omega-1,\\
    \Gamma\mathbbm{1}^{s^{\{S\}}} &= 0^{s^{\{S\}}},
\end{align}
\end{subequations}
where $0^{s^{\{S\}}}$ is a vector of zeros with length $s^{\{S\}}$.
\end{theorem}
\begin{proof}
Computing the fast GARK abscissae from the respective fast GARK tables,
\begin{align*}
    \bc^{\{F,F\}} &= \bA^{\{F,F\}}\mathbbm{1}^{s^{\{SF\}}} = (C^{\{S\}}\otimes A^{\{F\}})\mathbbm{1}^{s^{\{SF\}}}
    = c^{\{S\}}\otimes c^{\{F\}}, \\
    \bc^{\{F,E\}} &= \bA^{\{F,E\}}\mathbbm{1}^{s^{\{S\}}} = \sum_{k=0}^{n_\Omega-1}\Omega^{\{k\}}\otimes A^{\{F\}}c^{\{F\}\times k}\mathbbm{1}^{s^{\{S\}}} \\
    &= (\Omega^{\{0\}}\mathbbm{1}^{s^{\{S\}}})\otimes c^{\{F\}} + \sum_{k=1}^{n_\Omega-1}(\Omega^{\{k\}}\mathbbm{1}^{s^{\{S\}}})\otimes A^{\{F\}}c^{\{F\}\times k}, \\
    \bc^{\{F,I\}} &= \bA^{\{F,I\}}\mathbbm{1}^{s^{\{S\}}} = \bA^{\{F,E\}}\mathbbm{1}^{s^{\{S\}}} = \bc^{\{F,E\}},
\end{align*}
thus $\bc^{\{F,F\}}=\bc^{\{F,I\}}=\bc^{\{F,E\}}$ when $\Omega^0\mathbbm{1}^{s^{\{S\}}}=c^{\{S\}}$ and $\Omega^{\{k\}}\mathbbm{1}^{s^{\{S\}}}=0^{s^{\{S\}}},\ k=1,...,n_\Omega-1$.
Computing the slow GARK abscissae from the slow GARK tables,
\begin{align*}
    \bc^{\{S,F\}} &= \bA^{\{S,F\}}\mathbbm{1}^{s^{\{SF\}}} = (C^{\{S\}}\otimes b^{\{F\}})\mathbbm{1}^{s^{\{SF\}}} = c^{\{S\}}, \\
    \bc^{\{S,E\}} &= \bA^{\{S,E\}}\mathbbm{1}^{s^{\{S\}}} = \overline{\Omega}\mathbbm{1}^{s^{\{S\}}} = \sum_{k=0}^{n_\Omega-1}\Omega^{\{k\}}\mathbbm{1}^{s^{\{S\}}} \frac{1}{k+1}, \\
    \bc^{\{S,E\}} &= \bA^{\{S,I\}}\mathbbm{1}^{s^{\{S\}}} = (\overline{\Omega} + \Gamma)\mathbbm{1}^{s^{\{S\}}} = \sum_{k=0}^{n_\Omega-1}\Omega^{\{k\}}\mathbbm{1}^{s^{\{S\}}} \frac{1}{k+1} + \Gamma\mathbbm{1}^{s^{\{S\}}},
\end{align*}
and thus $\bc^{\{S,F\}}=\bc^{\{S,I\}}=\bc^{\{S,E\}}$ is satisfied using the same conditions as above, with the additional constraint that $\Gamma\mathbbm{1}^{s^{\{S\}}}=0^{s^{\{S\}}}$.
\end{proof}

{\remark[Second order conditions]{When the tables that comprise a GARK method are each at least second-order accurate and internal consistency holds, there are no additional second-order coupling conditions. Thus, the internal consistency conditions act as second-order conditions when the slow base method and arbitrary fast method are at least second-order accurate.}}

\subsection{Higher Order Conditions}
\label{sec:higherorderconditions}

\begin{theorem}[Third Order Conditions]
An internally-consistent IMEX-MRI-SR method with third-order accurate slow base method and with fast method of order $\max(3,n_\Omega+1)$ accurate is third-order accurate if the following condition holds:
\begin{equation}
    \label{eq:order3simplified}
    e_{s^{\{S\}}}^T\left(\sum_{k=0}^{n_\Omega-1}\Omega^{\{k\}}\frac{1}{(k+1)(k+2)}\right)c^{\{S\}}=\frac{1}{6}.
\end{equation}
\label{thm:thirdorder}
\end{theorem}
\begin{proof} From \cite{chinomona_implicit-explicit_2021}, a internally consistent GARK method of this structure with a third-order accurate slow base method and an order $\max(3,n_\Omega+1)$ fast method has four third-order coupling conditions,
\begin{subequations}
    \begin{align}
        \label{eq:order3abase}
        \bb^{\{\sigma\},T}\bA^{\{S,F\}}\bc^{\{F\}} &= \frac16,\\
        \label{eq:order3bbase}
         \bb^{\{F\},T}\bA^{\{S,\sigma\}}\bc^{\{S\}} &= \frac16,
    \end{align}
\end{subequations}
for $\sigma\in\{I,E\}$. An internally-consistent IMEX-MRI-SR method has
\begin{subequations}
    \begin{align}
        \label{eq:cFdef}
        \bc^{\{F\}}&=\bc^{\{F,F\}}=\bc^{\{F,E\}}=\bc^{\{F,I\}}=c^{\{S\}}\otimes c^{\{F\}},\\
        \label{eq:cSdef}
        \bc^{\{S\}}&=\bc^{\{S,F\}}=\bc^{\{S,E\}}=\bc^{\{S,I\}}=c^{\{E\}}=c^{\{I\}}=c^{\{S\}}.
    \end{align}
\end{subequations}
Conditions \eqref{eq:order3abase} are automatically satisfied for both values of $\sigma$:
\begin{align*}
    \frac16 &= \bb^{\{\sigma\},T}\bA^{\{S,F\}}\bc^{\{F\}} \\
    &= b^{\{\sigma\},T}(C^{\{S\}}\otimes  b^{\{F\},T})(c^{\{S\}}\otimes c^{\{F\}}) = \frac{1}{3}\cdot\frac{1}{2}
\end{align*}
Conditions \eqref{eq:order3bbase} reduce to the single condition \eqref{eq:order3simplified} because $\bA^{\{F,E\}}=\bA^{\{F,I\}}$ for an IMEX-MRI-SR method:
\begin{align*}
    \frac16 &= \bb^{\{F\},T}\bA^{\{F,\sigma\}}\bc^{\{S\}} = (e_{s^{\{S\}}}^T\otimes b^{\{F\},T})\left(\sum_{k=0}^{n_\Omega-1}\Omega^{\{k\}}\otimes A^{\{F\}}c^{\{F\}\times k}\right) c^{\{S\}} \\
    &= e_{s^{\{S\}}}^T\left(\sum_{k=0}^{n_\Omega-1}\Omega^{\{k\}}\frac{1}{(k+1)(k+2)}\right) c^{\{S\}}.
\end{align*}
\end{proof}

\begin{theorem}[Fourth Order Conditions]
An IMEX-MRI-SR method satisfying Theorem \ref{thm:thirdorder} is fourth-order accurate if the slow base method is fourth-order accurate, the arbitrary fast method is order $\max(4,n_\Omega+2)$ accurate, and the following conditions hold:
\begin{subequations}
    \label{eq:order4simplified}
    \begin{align}
        \label{eq:order4simplifiedA}
        e_{s^{\{S\}}}^T\left(\sum_{k=0}^{n_\Omega-1}\Omega^{\{k\}}\frac{1}{(k+1)(k+3)}\right)c^{\{S\}} &= \frac18, \\
    \label{eq:order4simplifiedB}
        e_{s^{\{S\}}}^T\left(\sum_{k=0}^{n_\Omega-1}\Omega^{\{k\}} \frac{1}{(k+1)(k+2)}\right)C^{\{S\}}c^{\{S\}} &= \frac{1}{12}, \\
    \label{eq:order4simplifiedC}
     e_{s^{\{S\}}}^T\Gamma C^{\{S\}}\left(\sum_{k=0}^{n_\Omega-1}\Omega^{\{k\}}\frac{1}{(k+1)(k+2)}\right)c^{\{S\}} &= 0, \\
    \label{eq:order4simplifiedD}
        e_{s^{\{S\}}}^T\overline{\Omega} C^{\{S\}}\left(\sum_{k=0}^{n_\Omega-1}\Omega^{\{k\}}\frac{1}{(k+1)(k+2)}\right)c^{\{S\}} &= \frac{1}{24},\\
    \label{eq:order4simplifiedF}
        e_{s^{\{S\}}}^T\left(\sum_{k=0}^{n_\Omega-1}\Omega^{\{k\}}\frac{1}{(k+1)(k+2)}\right)\overline{\Omega}c^{\{S\}}&=\frac{1}{24}, \\
    \label{eq:order4simplifiedG}
        e_{s^{\{S\}}}^T\left(\sum_{k=0}^{n_\Omega-1}\Omega^{\{k\}}\frac{1}{(k+1)(k+2)}\right)\Gamma c^{\{S\}}&=0.
    \end{align}
\end{subequations}
\label{thm:fourthorder}
\end{theorem}
\begin{proof}
From \cite{chinomona_implicit-explicit_2021}, there are 26 fourth-order coupling conditions for a third-order GARK method with this structure; this further reduces to 21 when $\bA^{\{F,E\}}=\bA^{\{F,I\}}$. These GARK coupling conditions are:
\begin{subequations}
\label{eq:cond4}
\begin{align}
    \label{eq:cond4a}
    \bb^{\{\sigma\},T} \bC^{\{S\}}\bA^{\{S,F\}}\bc^{\{F\}}&=\frac18 \\
    \label{eq:cond4b}
    \bb^{\{\sigma\},T}\bA^{\{S,\nu\}}\bA^{\{S,F\}}\bc^{\{F\}}&=\frac{1}{24} \\
    \label{eq:cond4c}
    \bb^{\{\sigma\},T}\bA^{\{S,F\}}\bC^{\{F\}}\bc^{\{F\}} &= \frac{1}{12} \\
    \label{eq:cond4d}
    \bb^{\{\sigma\},T}\bA^{\{S,F\}}\bA^{\{F,F\}}\bc^{\{F\}} &= \frac{1}{24}\\
    \label{eq:cond4e}
    \bb^{\{F\},T} \bC^{\{F\}}\bA^{\{F,\sigma\}}\bc^{\{S\}}&=\frac18 \\
    \label{eq:cond4f}
    \bb^{\{F\},T}\bA^{\{F,\sigma\}}\bC^{\{S\}}\bc^{\{S\}}&= \frac{1}{12} \\
    \label{eq:cond4g}
    \bb^{\{\sigma\},T}\bA^{\{S,F\}}\bA^{\{F,\nu\}}\bc^{\{S\}}&=\frac{1}{24} \\
    \label{eq:cond4h}
    \bb^{\{F\},T}\bA^{\{F,F\}}\bA^{\{F,\sigma\}}\bc^{\{S\}}&=\frac{1}{24} \\
    \label{eq:cond4i}
    \bb^{\{F\},T}\bA^{\{F,\sigma\}}\bA^{\{S,\nu\}}\bc^{\{S\}}&= \frac{1}{24} \\
    \label{eq:cond4j}
    \bb^{\{F\},T}\bA^{\{F,\sigma\}}\bA^{\{S,F\}}\bc^{\{F\}}&=\frac{1}{24}
\end{align}
\end{subequations}
for $\sigma,\nu\in\{I,E\}$. $\bc^{\{F\}}$ and $\bc^{\{S\}}$ are defined as in \eqref{eq:cFdef} and \eqref{eq:cSdef}, respectively. We arrive at the conditions \eqref{eq:order4simplified} by checking each of the conditions \eqref{eq:cond4} in turn.

The first four conditions \eqref{eq:cond4a}-\eqref{eq:cond4d} are automatically satisfied:
\begin{align*}
    \frac18 &= \bb^{\{\sigma\},T} \bC^{\{S\}}\bA^{\{S,F\}}\bc^{\{F\}} \\
    &= b^{\{\sigma\},T} C^{\{S\}}(C^{\{S\}}\otimes b^{\{F\},T})(c^{\{S\}}\otimes c^{\{F\}})\\
    &=b^{\{\sigma\},T} C^{\{S\}}C^{\{S\}}c^{\{S\}}\frac{1}{2} = \frac{1}{4}\cdot\frac{1}{2},\\
    \frac{1}{24} &= \bb^{\{\sigma\},T}\bA^{\{S,\nu\}}\bA^{\{S,F\}}\bc^{\{F\}} \\
    &= b^{\{\sigma\},T}A^{\{\nu\}}(C^{\{S\}}\otimes b^{\{F\},T})(c^{\{S\}}\otimes c^{\{F\}}) \\
    &=b^{\{\sigma\},T}A^{\{\nu\}}C^{\{S\}}c^{\{S\}}\frac{1}{2} = \frac{1}{12}\cdot\frac{1}{2},\\
    \frac{1}{12} &= \bb^{\{\sigma\},T}\bA^{\{S,F\}}\bC^{\{F\}}\bc^{\{F\}} \\
    &= b^{\{\sigma\},T}(C^{\{S\}}\otimes b^{\{F\},T})\operatorname{diag}(c^{\{S\}}\otimes c^{\{F\}})(c^{\{S\}}\otimes c^{\{F\}}) \\
    &= b^{\{\sigma\},T}C^{\{S\}}C^{\{S\}}c^{\{S\}}\frac{1}{3} = \frac{1}{4}\cdot\frac{1}{3},\\
    \frac{1}{24} &= \bb^{\{\sigma\},T}\bA^{\{S,F\}}\bA^{\{F,F\}}\bc^{\{F\}} \\
    &= b^{\{\sigma\},T}(C^{\{S\}}\otimes b^{\{F\},T})(C^{\{S\}}\otimes A^{\{F\}})(c^{\{S\}}\otimes c^{\{F\}}) \\
    &= b^{\{\sigma\},T}C^{\{S\}}C^{\{S\}}c^{\{S\}} \frac{1}{6} = \frac{1}{4}\cdot\frac{1}{6}.
\end{align*}

Condition \eqref{eq:cond4e} is not automatically satisfied and simplifies to \eqref{eq:order4simplifiedA}:
\begin{align*}
    \frac{1}{8} &= \bb^{\{F\},T} \bC^{\{F\}}\bA^{\{F,\sigma\}}\bc^{\{S\}} \\
    &= (e_{s^{\{S\}}}\otimes b^{\{F\}})^T \operatorname{diag}(c^{\{S\}}\otimes c^{\{F\}})\left(\sum_{k=0}^{n_\Omega-1}\Omega^{\{k\}}\otimes A^{\{F\}}c^{\{F\}\times k}\right)c^{\{S\}} \\
    &= e_{s^{\{S\}}}^T\left(\sum_{k=0}^{n_\Omega-1}\Omega^{\{k\}}\frac{1}{(k+1)(k+3)}\right)c^{\{S\}}.
\end{align*}

Condition \eqref{eq:cond4f} is not automatically satisfied and simplifies to \eqref{eq:order4simplifiedB}.
\begin{align*}
    \frac{1}{12} &= \bb^{\{F\},T}\bA^{\{F,\sigma\}}\bC^{\{S\}}\bc^{\{S\}} \\
    &= (e_{s^{\{S\}}}\otimes b^{\{F\}})^T\left(\sum_{k=0}^{n_\Omega-1}\Omega^{\{k\}}\otimes A^{\{F\}}c^{\{F\}\times k}\right)C^{\{S\}}c^{\{S\}} \\
    &= e_{s^{\{S\}}}^T\left(\sum_{k=0}^{n_\Omega-1}\Omega^{\{k\}} \frac{1}{(k+1)(k+2)}\right)C^{\{S\}}c^{\{S\}}.
\end{align*}

Condition \eqref{eq:cond4g} is not automatically satisfied and reduces to \eqref{eq:order4simplifiedC} and \eqref{eq:order4simplifiedD} when $\sigma=I,E$, respectively, and both conditions are enforced simultaneously,
\begin{align*}
    \frac{1}{24} &= \bb^{\{\sigma\},T}\bA^{\{S,F\}}\bA^{\{F,\nu\}}\bc^{\{S\}} \\
    &= b^{\{\sigma\},T}(C^{\{S\}}\otimes b^{\{F\}})\left(\sum_{k=0}^{n_\Omega-1}\Omega^{\{k\}}\otimes A^{\{F\}}c^{\{F\}\times k}\right)c^{\{S\}} \\
    &= b^{\{\sigma\},T}C^{\{S\}}\left(\sum_{k=0}^{n_\Omega-1}\Omega^{\{k\}}\frac{1}{(k+1)(k+2)}\right)c^{\{S\}}.
\end{align*}

Condition \eqref{eq:cond4h} is not automatically satisfied but simplifies to \eqref{eq:order3simplified} minus \eqref{eq:order4simplifiedA},
\begin{align*}
    \frac{1}{24} &= \bb^{\{F\},T}\bA^{\{F,F\}}\bA^{\{F,\sigma\}}\bc^{\{S\}} \\
    &= (e_{s^{\{S\}}}\otimes b^{\{F\}})^T(C^{\{S\}}\otimes A^{\{F\}})\left(\sum_{k=0}^{n_\Omega-1}\Omega^{\{k\}}\otimes A^{\{F\}}c^{\{F\}\times k}\right)c^{\{S\}} \\
    &= e_{s^{\{S\}}}^T\left(\sum_{k=0}^{n_\Omega-1}\Omega^{\{k\}}\frac{1}{(k+1)(k+2)(k+3)}\right)c^{\{S\}}.
\end{align*}

Condition \eqref{eq:cond4i} is not automatically satisfied and reduces to \eqref{eq:order4simplifiedF} and \eqref{eq:order4simplifiedG} when $\nu=I,E$, respectively, and both conditions are enforced simultaneously:
\begin{align*}
    \frac{1}{24} &= \bb^{\{F\},T}\bA^{\{F,\sigma\}}\bA^{\{S,\nu\}}\bc^{\{S\}} \\
    &= (e_{s^{\{S\}}}\otimes b^{\{F\}})^T\left(\sum_{k=0}^{n_\Omega-1}\Omega^{\{k\}}\otimes A^{\{F\}}c^{\{F\}\times k}\right)\bA^{\{S,\nu\}}c^{\{S\}} \\
    &= e_{s^{\{S\}}}^T\left(\sum_{k=0}^{n_\Omega-1}\Omega^{\{k\}}\frac{1}{(k+1)(k+2)}\right)\bA^{\{S,\nu\}}c^{\{S\}}.
\end{align*}

Finally, condition \eqref{eq:cond4j} simplifies to the same condition as \eqref{eq:order4simplifiedB},
\begin{align*}
    \frac{1}{24} &= \bb^{\{F\},T}\bA^{\{F,\sigma\}}\bA^{\{S,F\}}\bc^{\{F\}} \\
    &= (e_{s^{\{S\}}}\otimes b^{\{F\}})\left(\sum_{k=0}^{n_\Omega-1}\Omega^{\{k\}}\otimes A^{\{F\}}c^{\{F\}\times k}\right)(C^{\{S\}}\otimes b^{\{F\}})(c^{\{S\}}\otimes c^{\{F\}}) \\
    &= e_{s^{\{S\}}}^T\left(\sum_{k=0}^{n_\Omega-1}\Omega^{\{k\}}\frac{1}{(k+1)(k+2)}\right)C^{\{S\}}c^{\{S\}}\frac{1}{2}.
\end{align*}
\end{proof}

\begin{theorem}[Minimum $n_\Omega$ for third-order accuracy]
An IMEX-MRI-SR requires at least $n_\Omega=2$ for third-order accuracy.
\end{theorem}
\begin{proof}
An IMEX-MRI-SR method with one $\Omega$ matrix, $\Omega^{\{0\}}$, has $\bA^{\{S,E\}}=\overline{\Omega}=\Omega^0$ and, from the FSAL property, $\bb^{\{E\},T}=e_{s^{\{S\}}}^T\overline{\Omega}=e_{s^{\{S\}}}^T\Omega^{\{0\}}$. For an IMEX-MRI-SR method of order $p$, we assume the slow-explicit base method also satisfies all standard Runge--Kutta order conditions up to order $p$. Thus, a third-order IMEX-MRI-SR slow-explicit base method with one $\Omega$ satisfies the second order condition $b^{\{E\},T}c^{\{S\}}=\frac{1}{2}$, which simplifies to
\[
  \frac{1}{2}=\sum_{j=1}^{s^{\{S\}}}\omega^{\{0\}}_{s^{\{S\}},j}c^{\{S\}}_j.
\]
However, the third-order IMEX-MRI-SR coupling condition \eqref{eq:order3simplified} simplifies to
\[
  \frac{1}{3}=\sum_{j=1}^{s^{\{S\}}}\omega^{\{0\}}_{s^{\{S\}},j}c^{\{S\}}_j.
\]
Since these cannot hold simultaneously, a third-order IMEX-MRI-SR method with one $\Omega$ matrix is not possible.  There are no such mutually exclusive conditions for $n_\Omega>1$, and in Section \ref{sec:examplemethods} we introduce third and a fourth order methods with $n_\Omega=2$.
\end{proof}


\section{Linear Stability}
\label{sec:linearstability}

As in \cite{chinomona_implicit-explicit_2021,sandu_class_2019}, we analyze the stability of IMEX-MRI-SR methods when applied to the linear, scalar test problem
\begin{align}
    \label{eq:scalartest}
    y'(t) = \lambda^{\{F\}}y + \lambda^{\{E\}}y + \lambda^{\{I\}}y, \quad t\ge 0, \quad y(0)=1,
\end{align}
with $\lambda^{\{F\}},\lambda^{\{E\}},\lambda^{\{I\}}\in\mathbb{C}^-$. For convenience, we additionally define  $z^{\{F\}}=H\lambda^{\{F\}}$, $z^{\{E\}}=H\lambda^{\{E\}}$, and $z^{\{I\}}=H\lambda^{\{I\}}$.  When applied to \eqref{eq:scalartest}, the IVP \eqref{eq:fastivp} becomes
\begin{align*}
    v_i'(\theta) &= \lambda^{\{F\}}v_i(\theta) + \frac{1}{c^{\{S\}}_i}\sum_{j=1}^{i-1} \omega_{i,j}\left(\frac{\theta}{c^{\{S\}}_iH}\right)(\lambda^{\{E\}}Y_j^{\{S\}}+\lambda^{\{I\}}Y_j^{\{S\}}) \\
    &= \lambda^{\{F\}}v_i(\theta) + \frac{1}{c^{\{S\}}_i}\sum_{j=1}^{i-1}\sum_{k=0}^{n_\Omega-1} \omega^{\{k\}}_{i,j}\left(\frac{\theta}{c^{\{S\}}_iH}\right)^k(\lambda^{\{E\}} Y_j^{\{S\}} + \lambda^{\{I\}} Y_j^{\{S\}})
\end{align*}
for $i=2,...,s^{\{S\}}$, with $\theta\in[0,c^{\{S\}}_iH]$ and $v_i(0)=y_n$. The solution to this at $\theta = c^{\{S\}}_iH$ is
\begin{align*}
    v_i(c^{\{S\}}_iH) &= e^{c^{\{S\}}_iz^{\{F\}}}y_n + (z^{\{E\}}+z^{\{I\}})\sum_{j=1}^{i-1} \sum_{k=0}^{n_\Omega-1} \omega^{\{k\}}_{i,j}\left(\int_0^1e^{c^{\{S\}}_iz^{\{F\}}(1-t)}t^k \mathrm dt\right) Y_j^{\{S\}} \\
    &= \varphi_0(c^{\{S\}}_iz^{\{F\}})y_n + (z^{\{E\}}+z^{\{I\}})\sum_{j=1}^{i-1}\eta_{i,j}(z^{\{F\}}) Y_j^{\{S\}}.
\end{align*}
Here we define $\eta$ as in \cite{chinomona_implicit-explicit_2021} as a function of the fast eigenvalue $z^{\{F\}}$,
\begin{align}
  \eta_{i,j}(z^{\{F\}}) = \sum_{k=0}^{n_\Omega-1}\omega^{\{k\}}_{i,j}\varphi_{k+1}(c^{\{S\}}_iz^{\{F\}}),
\end{align}
where the family of analytical functions $\{\varphi_k\}$ are given by \cite{sandu_class_2019},
\begin{align*}
    \varphi_0(z)=e^z,\quad \varphi_k(z)=\int_0^1 e^{z(1-t)}t^{k-1} \mathrm dt,\quad k\ge 1.
\end{align*}
The $i^{th}$ IMEX-MRI-SR stage for the stability problem \eqref{eq:scalartest} then becomes
\begin{align}
    \label{eq:stabilityslowstage}
    Y_i^{\{S\}} &= \varphi_0(c^{\{S\}}_iz^{\{F\}})y_n + (z^{\{E\}}+z^{\{I\}})\sum_{j=1}^{i-1}\eta_{i,j}(z^{\{F\}})Y_j^{\{S\}} + z^{\{I\}}\sum_{j=1}^{i}\gamma_{i,j}Y_j^{\{S\}}.
\end{align}
Concatenating $Y=[Y_1^{\{S\}}\ \cdots\ Y^{\{S\}}_{s^{\{S\}}}]^T$ and writing \eqref{eq:stabilityslowstage} in matrix form,
\begin{align*}
    Y &= \varphi_0(c^{\{S\}}z^{\{F\}})y_n + (z^{\{E\}}+z^{\{I\}})\eta(z^{\{F\}})Y + z^{\{I\}}\Gamma Y \\
    &= \left(I-(z^{\{E\}}+z^{\{I\}})\eta(z^{\{F\}}) - z^{\{I\}}\Gamma\right)^{-1}\varphi_0(c^{\{S\}}z^{\{F\}})y_n
\end{align*}
where
\[
  \eta(z^{\{F\}}) = \sum_{k=0}^{n_\Omega-1}\operatorname{diag}(\varphi_{k+1}(c^{\{S\}}z^{\{F\}}))\Omega^{\{k\}}.
\]
Thus the stability function for an IMEX-MRI-SR method applied to \eqref{eq:scalartest} is
\begin{equation}
\begin{split}
    R(z^{\{F\}},&z^{\{E\}},z^{\{I\}}) = \\
    &e_{s^{\{S\}}}^T\left(I-(z^{\{E\}}+z^{\{I\}})\eta(z^{\{F\}}) - z^{\{I\}}\Gamma\right)^{-1}\varphi_0(c^{\{S\}}z^{\{F\}}).
\end{split}
\end{equation}

We consider a few definitions of joint stability for IMEX-MRI-SR methods. As with IMEX-MRI-GARK methods, we consider a region that incorporates all three $z$,
\begin{align}
    \label{eq:triplyjointstab}
    \mathcal{J}_{\alpha,\rho,\beta,\xi}=\{z^{\{E\}}\in\mathbb{C}^-:|R(z^{\{F\}},z^{\{E\}},z^{\{I\}})|\le1,\ \forall z^{\{F\}}\in\mathcal{S}^{\{F\}}_{\alpha,\rho},\ \forall z^{\{I\}}\in\mathcal{S}^{\{I\}}_{\beta,\xi}\},
\end{align}
where $\mathcal{S}^{\{\sigma\}}_{\alpha,\rho}=\{z^{\{\sigma\}}\in\mathbb{C}^-:|\operatorname{arg}(z^{\{\sigma\}})-\pi|\le\alpha,\ |z^{\{\sigma\}}|\le\rho\}$.
We note that this definition of joint stability may be overly-restrictive, as it demands the implicit and fast parts of the method to be A($\beta$)- and A($\alpha$)-stable, respectively, for any joint stability region to exist. Thus, we also analyze the implicit and explicit stability regions of IMEX-MRI-SR methods independently, with stability regions defined as:
\begin{align}
    \label{eq:exdoublejointstab}
    \mathcal{J}^{\{I\}}_{\alpha,\rho}&=\{z^{\{I\}}\in\mathbb{C}^-:|R(z^{\{F\}},0,z^{\{I\}})|\le1,\ \forall z^{\{F\}}\in\mathcal{S}^{\{F\}}_{\alpha,\rho}\}, \\
    \label{eq:imdoublejointstab}
    \mathcal{J}^{\{E\}}_{\alpha,\rho}&=\{z^{\{E\}}\in\mathbb{C}^-:|R(z^{\{F\}},z^{\{E\}},0)|\le1,\ \forall z^{\{F\}}\in\mathcal{S}^{\{F\}}_{\alpha,\rho}\}.
\end{align}
These are more consistent with standard stability analyses of ARK methods, wherein explicit and implicit stability are considered separately.

\section{Example Methods}
\label{sec:examplemethods}

We introduce three IMEX-MRI-SR methods of orders 2, 3 and 4, each of which includes an embedding with accuracy one order lower (1, 2 and 3, resp.) for temporal error estimation.  We note that we designed the embedding coefficients with efficiency in mind, so that computation of the embedded solutions do not require an additional implicit solve at the slow time scale (i.e., $\hat{\gamma}_{s^{\{S\}}} = 0$).

When presenting the coefficients for each method we use the notation
\begin{align*}
    \Omega^{\{k\}} =
    \begin{bmatrix}
    {\bf \Omega^{\{k\}}} \\ \hline
    \hat{\omega}^{\{k\}}
    \end{bmatrix},\ \Gamma =
    \begin{bmatrix}
    {\bf \Gamma} \\ \hline
    \hat{\gamma}
    \end{bmatrix},
\end{align*}
where  ${\bf \Omega^{\{k\}}}$ and ${\bf \Gamma}$ are the matrices of coefficients defining the primary method, and $\hat{\omega}^{\{k\}}$ and $\hat{\gamma}$ are the embedding row of coefficients.

\subsection{IMEX-MRI-SR2(1)}
\label{sec:secondordermethod}

The first method is second-order with a first-order embedding.  It has 4 stages, $n_\Omega=1$, and requires 3 slow nonlinear solves per step. The coefficients can be found in Appending \ref{appendix:imexmrisr21}.

To create this method, we used the free coefficients of the primary method to maximize the size of the stability region defined by \eqref{eq:triplyjointstab}.  We used the free coefficients of the embedded method to minimize the value of the C-statistic from Prince and Dormand \cite{PRINCE198167},
\begin{align}
    \label{eq:cstatistic}
    C^{(p+1)}=\frac{\|\hat{\tau}^{(p+1)}-\tau^{(p+1)}\|_2}{\|\hat{\tau}^{(p)}\|_2},
\end{align}
where $\tau^{(p)}$ is the vector of order condition residuals at order $p$ for the primary method, $\hat{\tau}^{(p)}$ is the vector of order condition residuals at order $p$ for the embedded method, and $p$ is the order of the primary method (in this case, $p=2$). The C-statistic gives an estimate of how much the error stemming from unsatisfied $(p+1)$-order conditions of the primary method corrupts the error estimate provided by the embedded method.

This method has large, robust stability regions. Figures \ref{fig:imexmrisr2110fangletriplejoint} and \ref{fig:imexmrisr2145fangletriplejoint} show $\mathcal{J}_{\alpha,10^2,\beta,10^4}$ for $\alpha\in\{10^\circ,45^\circ\}$ and varying $\beta$, along with the explicit slow base method's stability region.  We see that the multirate method has a stability region essentially identical to the explicit slow base method when $\alpha=10^\circ$ for any value of $z^{\{I\}}$. When $\alpha$ grows to $45^\circ$ we can see a decay of the stability region associated with the $\beta=85^\circ$, while the regions associated with smaller $\beta$ values experience negligible decay, if any.

Figure \ref{fig:imexmrisr21explicitdoublejoint} shows $\mathcal{J}^{\{E\}}_{\alpha,10^2}$ for varying $\alpha$. We see that when $z^{\{I\}}=0$, the stability region is again large, with only a slight decay in area when $\alpha=85^\circ$.

Figure \ref{fig:imexmrisr21explicitdoublejoint} shows $\mathcal{J}^{\{I\}}_{\alpha,10^2}$ for varying $\alpha$.  When $z^{\{E\}}=0$, the multirate method is A-stable for $0^\circ \le \alpha \le 45^\circ$. When $\alpha$ grows further to $65^\circ$, the region decays slightly but is still approximately A($80^\circ$)-stable. When $\alpha=85^\circ$, the stability region decays to an enclosed bubble similar to the regions in Figure \ref{fig:imexmrisr21explicitdoublejoint}.

\begin{figure}[htb]
\centering
\begin{subfigure}[b]{0.48\textwidth}
    \centering
    \includegraphics[width=\textwidth]{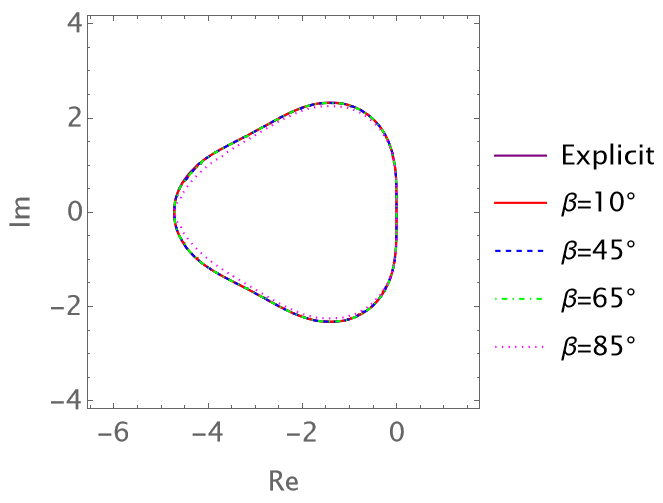}
    \caption{$\mathcal{J}_{10^\circ,10^2,\alpha,10^4}$}
    \label{fig:imexmrisr2110fangletriplejoint}
\end{subfigure}
\hfill
\begin{subfigure}[b]{0.48\textwidth}
    \centering
    \includegraphics[width=\textwidth]{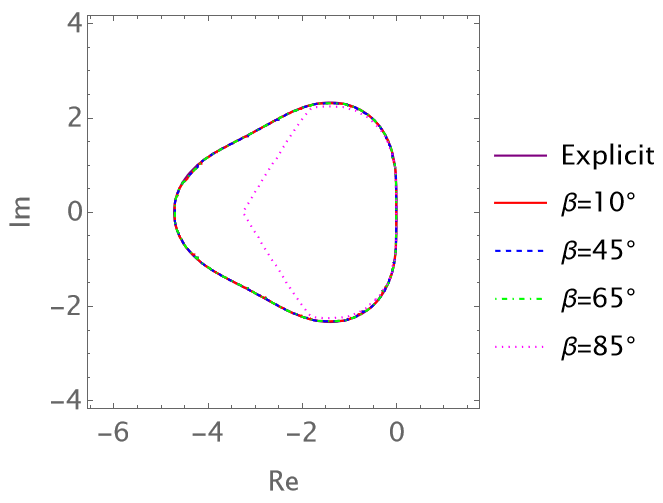}
    \caption{$\mathcal{J}_{45^\circ,10^2,\alpha,10^4}$}
    \label{fig:imexmrisr2145fangletriplejoint}
\end{subfigure}
\newline
\begin{subfigure}[b]{0.48\textwidth}
    \centering
    \includegraphics[width=\textwidth]{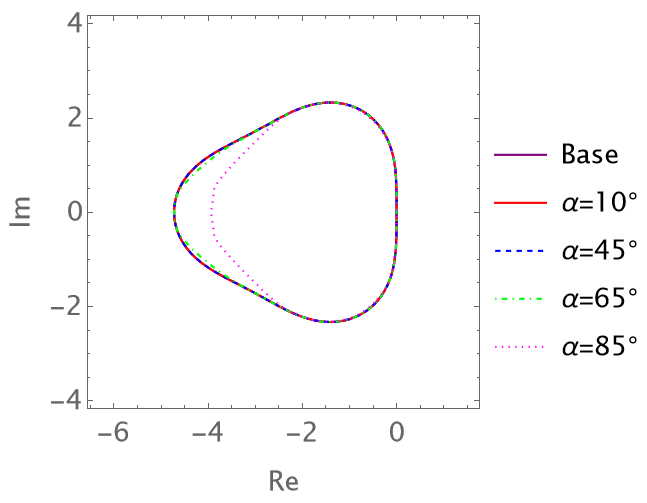}
    \caption{$\mathcal{J}^{\{E\}}_{\alpha,10^2}$}
    \label{fig:imexmrisr21explicitdoublejoint}
\end{subfigure}
\hfill
\begin{subfigure}[b]{0.48\textwidth}
    \centering
    \includegraphics[width=\textwidth]{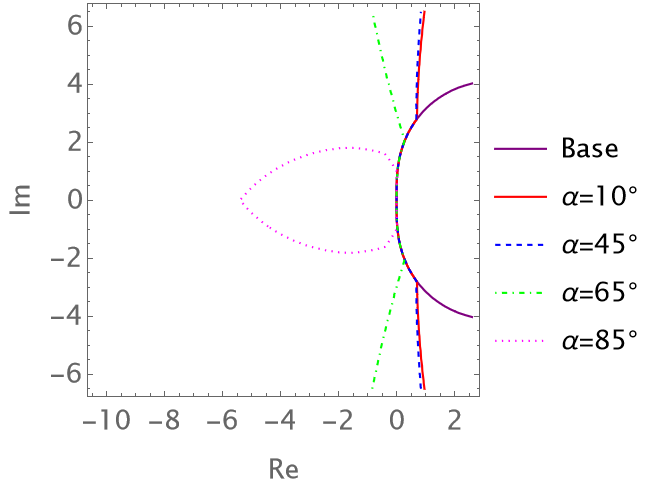}
    \caption{$\mathcal{J}^{\{I\}}_{\alpha,10^2}$}
    \label{fig:imexmrisr21implicitdoublejoint}
\end{subfigure}
\caption{Joint Stability Regions for IMEX-MRI-SR2(1)}
\label{fig:imexmrisr21stab}
\end{figure}

\subsection{IMEX-MRI-SR3(2)}
\label{sec:thirdordermethod}

Our second method is third-order with a second-order embedding.  It has 5 stages, $n_\Omega=2$, and requires 4 nonlinear solves per step. The coefficients can be found in Appendix \ref{appendix:imexmrisr32}.

To create this method, we used the free coefficients for both the method and embedding as described in Section \ref{sec:secondordermethod}, this time using the C-statistic \eqref{eq:cstatistic} with $p=3$.

Figures \ref{fig:imexmrisr3210fangletriplejoint} and \ref{fig:imexmrisr3245fangletriplejoint} show the relatively large joint stability regions $\mathcal{J}_{\alpha,10^2,\beta,10^4}$ for $\alpha=\{10^\circ,45^\circ\}$, respectively, along with the explicit slow base method's stability region.  Again, at $\alpha=10^\circ$ and lower values of $\beta$, the method is essentially as stable as the explicit slow base method. The areas of these stability regions decrease for higher values of $\beta$.  When $\alpha$ grows to $45^\circ$ in Figure \ref{fig:imexmrisr3245fangletriplejoint}, the stability regions shrink slightly in comparison with $\alpha=10^\circ$ from Figure \ref{fig:imexmrisr3210fangletriplejoint}.

Figures \ref{fig:imexmrisr32explicitdoublejoint} and \ref{fig:imexmrisr32implicitdoublejoint} show $\mathcal{J}^{\{E\}}_{\alpha,10^2}$ and $\mathcal{J}^{\{I\}}_{\alpha,10^2}$ for varying $\alpha$, respectively.  We see that when $z^{\{I\}}=0$, the stability region is reasonably large for smaller value of $\alpha$, but the region shrinks as $\alpha$ grows; however, even for $\alpha=85^\circ$, the region retains a good extent along the imaginary axis.  Similar to the second-order method, when $z^{\{E\}}=0$, the stability region is A-stable for most values of $\alpha$, only losing A-stability for $\alpha=85^\circ$, where the region decays to a rather large, yet enclosed, bubble.

\begin{figure}[htb]
\centering
\begin{subfigure}[b]{0.48\textwidth}
    \centering
    \includegraphics[width=\textwidth]{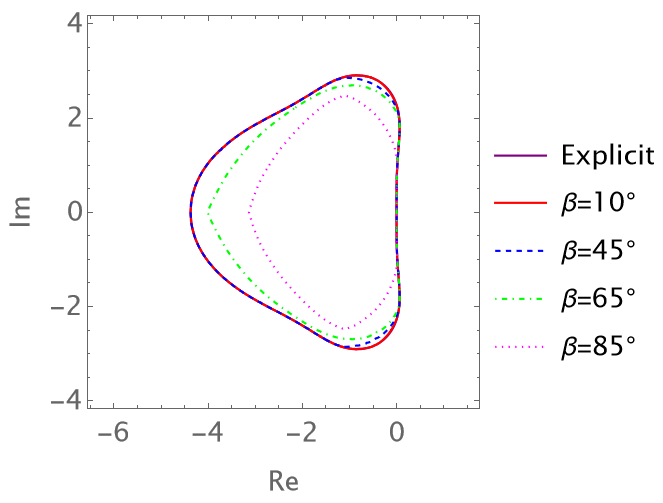}
    \caption{$\mathcal{J}_{10^\circ,10^2,\beta,10^4}$}
    \label{fig:imexmrisr3210fangletriplejoint}
\end{subfigure}
\hfill
\begin{subfigure}[b]{0.48\textwidth}
    \centering
    \includegraphics[width=\textwidth]{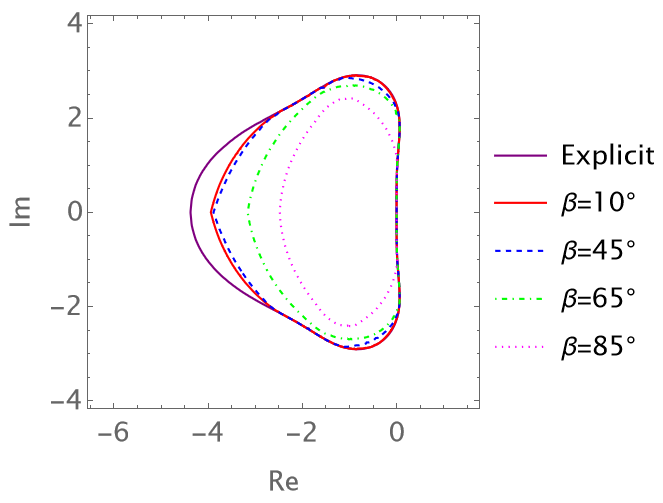}
    \caption{$\mathcal{J}_{45^\circ,10^2,\beta,10^4}$}
    \label{fig:imexmrisr3245fangletriplejoint}
\end{subfigure}
\newline
\begin{subfigure}[b]{0.48\textwidth}
    \centering
    \includegraphics[width=\textwidth]{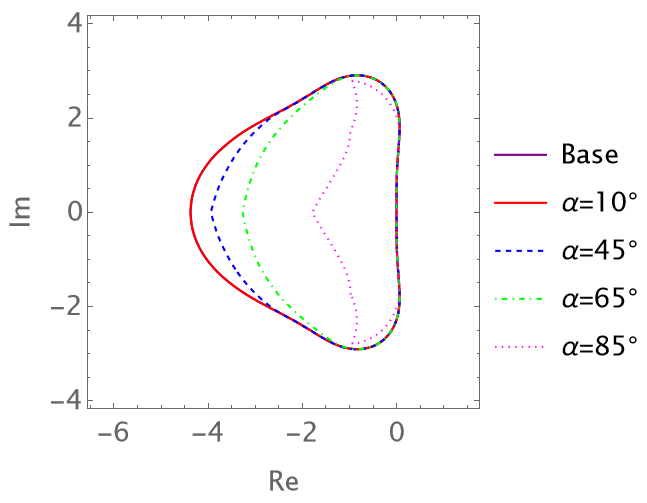}
    \caption{$\mathcal{J}^{\{E\}}_{\alpha,10^2}$}
    \label{fig:imexmrisr32explicitdoublejoint}
\end{subfigure}
\hfill
\begin{subfigure}[b]{0.48\textwidth}
    \centering
    \includegraphics[width=\textwidth]{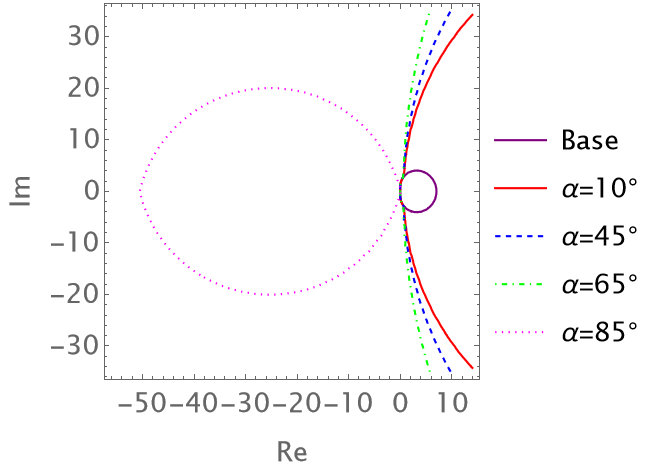}
    \caption{$\mathcal{J}^{\{I\}}_{\alpha,10^2}$}
    \label{fig:imexmrisr32implicitdoublejoint}
\end{subfigure}
\caption{Joint Stability Regions for IMEX-MRI-SR3(2)}
\label{fig:imexmrisr32stab}
\end{figure}

\subsection{IMEX-MRI-SR4(3)}
\label{sec:fourthordermethod}

Our final method is fourth-order with a third-order embedding, has 7 stages, $n_\Omega=2$, and requires 5 nonlinear solves per step. The coefficients can be found in Appendix \ref{appendix:imexmrisr43}.

Due to the number of order conditions involved when simultaneously solving all IMEX-MRI-SR coupling conditions and base ARK order conditions at fourth-order, we based this method off of an existing 4(3) ARK method, LIRK4 \cite{calvo2001linearly}.  Since this was not stiffly accurate, then as discussed in Remark \ref{remark:existingarkbase} we padded the $A^{\{E\}}$ and $A^{\{I\}}$ matrices with the $b$ vector, i.e.,
\begin{equation}
    \bA^{\{S,E\}}=\begin{bmatrix}
    A^{\{E\}} & 0^{\{6\}} \\
    b^T & 0
    \end{bmatrix},\ \bA^{\{S,I\}}=\begin{bmatrix}
    A^{\{I\}} & 0^{\{6\}} \\
    b^T & 0
    \end{bmatrix},
\end{equation}
where $0^{\{6\}}\in\mathbb{R}^6$ is all zero.

Since the last row of $\bA^{\{S,E\}}$ and $\bA^{\{S,I\}}$ both equal $b^T$, the last row of $\Gamma$ equals zero, as follows from the definitions \eqref{eq:ASEtable} and \eqref{eq:ASItable}. When $\bA^{\{S,I\}}$ has a zero in the bottom right entry (and therefore $\Gamma$ has a zero in the bottom right entry, from \eqref{eq:ASItable}), there is no nonlinear solve required to compute this last stage and therefore the updated time step solution.  We believe that this negatively affects stability, as we will show in Figure \ref{fig:imexmrisr43stab}.

As before, we used the free variables of the IMEX-MRI-SR method to optimize stability. This method had an empty joint stability region defined by \eqref{eq:triplyjointstab} for all of our attempts to choose or optimize values of the free variables, so we instead optimized the size of the stability region defined by \eqref{eq:imdoublejointstab}.  To optimize the embedded method, we minimized the 2-norm of the fourth-order condition residuals, $\|\hat{\tau}^{(4)}\|_2$, to reduce the overall error in the embedded solution.  We note that our previous approach of minimizing the C-statistic \eqref{eq:cstatistic} was not possible since we do not yet have the fifth-order IMEX-MRI-SR coupling conditions.

Figure \ref{fig:imexmrisr43explicitdoublejoint} shows the stability regions $\mathcal{J}^{\{E\}}_{\alpha,1}$ for varying $\alpha$. Unlike the lower-order methods, this method's explicit stability region never fully matches that of the explicit base method. Similarly to the other methods, as $\alpha$ grows the stability region shrinks.

Figure \ref{fig:imexmrisr43implicitdoublejoint} shows the stability regions $\mathcal{J}^{\{I\}}_{\alpha,1}$ for varying $\alpha$. Again unlike the lower-order methods, these regions are not A-stable for $\alpha \ne 0$.
%
%
Because $\mathcal{J}^{\{I\}}_{\alpha,\rho}$ is never A-stable for this method, the stability region \eqref{eq:triplyjointstab} is always empty.  We suspect that this is primarily due to the lack of an implicit solve in the last stage of the method.  In future work we plan to investigate this issue in more detail to better understand the factors that contribute to IMEX-MRI-SR joint stability.

\begin{figure}[htb]
\centering
\begin{subfigure}[b]{0.48\textwidth}
    \centering
    \includegraphics[width=\textwidth]{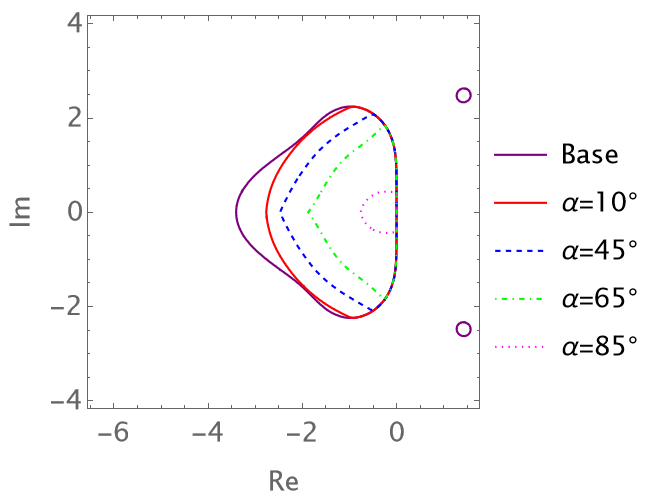}
    \caption{$\mathcal{J}^{\{E\}}_{\alpha,10^2}$}
    \label{fig:imexmrisr43explicitdoublejoint}
\end{subfigure}
\hfill
\begin{subfigure}[b]{0.48\textwidth}
    \centering
    \includegraphics[width=\textwidth]{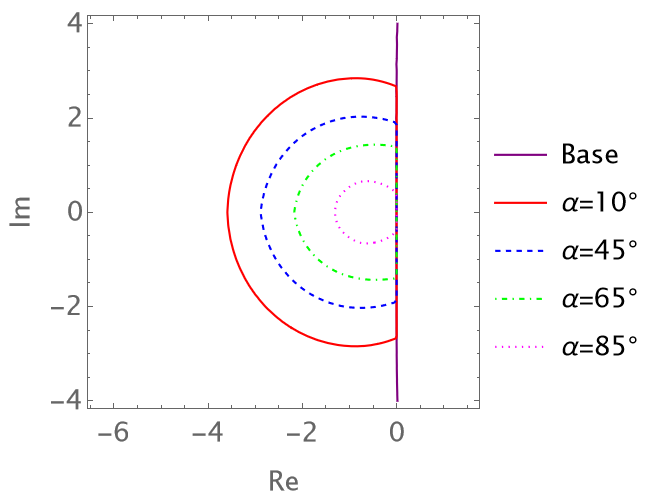}
    \caption{$\mathcal{J}^{\{I\}}_{\alpha,10^2}$}
    \label{fig:imexmrisr43implicitdoublejoint}
\end{subfigure}
\caption{Joint Stability Regions for IMEX-MRI-SR4(3)}
\label{fig:imexmrisr43stab}
\end{figure}

\section{MERK Methods as Explicit IMEX-MRI-SR Methods}
\label{sec:merkasimexmrisr}

In \cite{luan_new_2020}, Luan et al.~define MERK methods with orders of accuracy spanning two through five by explicitly defining the abscissae $c_i^{\{S\}}$ and the forcing functions $g_i(\theta)$.  Due to their similar structure to IMEX-MRI-SR methods, we may analyze MERK methods using our theory from Section \ref{sec:orderconditions}.  Because MERK methods are always explicit, their IMEX-MRI-SR $\Gamma$ matrices will be all zero.  We recall that MERK methods are defined under an assumption that the fast function is linear, $f^{\{F\}}(t,y) = \mathcal{L}y$, but that the slow function can be arbitrary.  It is thus natural to assume that MERK methods might only satisfy a subset of the IMEX-MRI-SR order conditions, potentially failing those that handle nonlinearity in $f^{\{F\}}$.

\subsection{MERK2}
\label{sec:merk2}
Converting the three-stage second-order accurate MERK2 method from \cite{luan_new_2020} into IMEX-MRI-SR form, we have
\begin{align}
    \Omega^{\{0\}} &= \begin{bmatrix}
    0 & 0 & 0  \\
    c^{\{S\}}_2 & 0 & 0 \\
    1& 0 & 0
    \end{bmatrix},\quad \Omega^{\{1\}} = \begin{bmatrix}
    0 & 0 & 0  \\
    0 & 0 & 0  \\
    -1/c^{\{S\}}_2 & 1/c^{\{S\}}_2 & 0
    \end{bmatrix}.
\end{align}
Interestingly, these coefficients (along with $\Gamma = 0$) satisfy \emph{all} coupling conditions up to third order, and the slow base method determined by $\overline{\Omega}$ satisfies all conditions up to order two. Thus, we expect MERK2 to have second-order accuracy, \emph{even for nonlinear $f^{\{F\}}$}.  We confirm this with numerical tests involving nonlinear $f^{\{F\}}$ in Section \ref{sec:numericalresults}, where we use $c^{\{S\}}_2=\frac12$ since it was unspecified in \cite{luan_new_2020}.

Figure \ref{fig:imexmrisrmerk2explicitdoublejoint} shows the stability regions for MERK2. Because this method is explicit, we only plot $\mathcal{J}^{\{E\}}_{\alpha,10^2}$.  Notably, these regions nearly match that of the base explicit method for most values of $\alpha$ examined.

\begin{figure}[htb]
\centering
\begin{subfigure}[b]{0.48\textwidth}
    \centering
    \includegraphics[width=\textwidth]{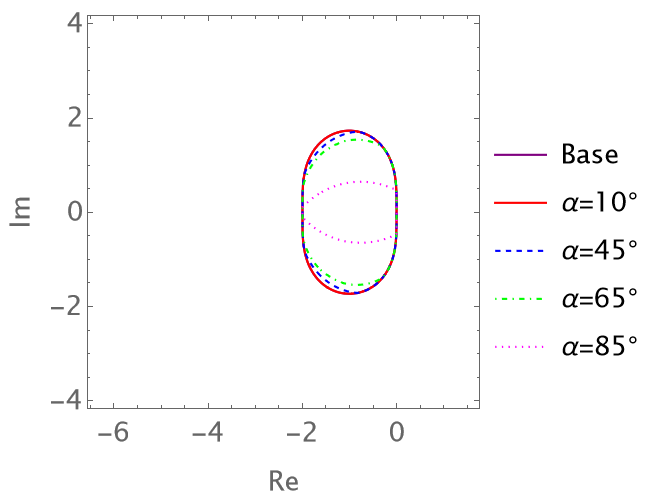}
    \caption{MERK2}
    \label{fig:imexmrisrmerk2explicitdoublejoint}
\end{subfigure}
\hfill
\begin{subfigure}[b]{0.48\textwidth}
    \centering
    \includegraphics[width=\textwidth]{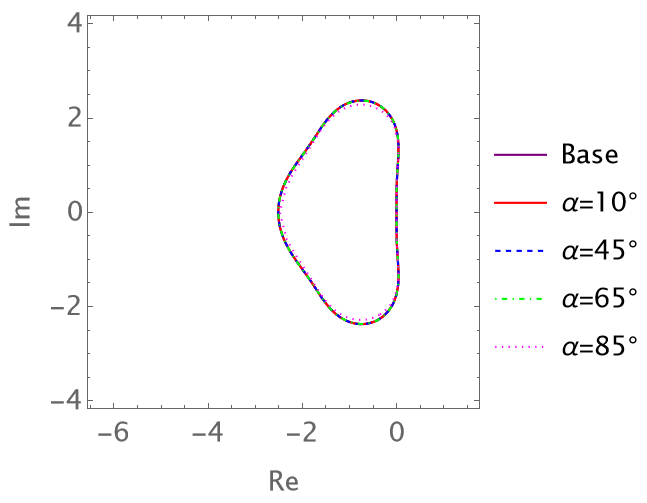}
    \caption{MERK3}
    \label{fig:imexmrisrmerk3explicitdoublejoint2}
\end{subfigure}
\caption{$\mathcal{J}^{\{E\}}_{\alpha,10^2}$ Regions for MERK2 and MERK3}
\label{fig:imexmrisrmerk23stab}
\end{figure}

\subsection{MERK3}
\label{sec:merk3}
Converting the four-stage, third-order accurate, MERK3 from \cite{luan_new_2020} to an IMEX-MRI-SR method, we have
\begin{align}
    \Omega^{\{0\}} &= \begin{bmatrix}
    0 & 0 & 0 & 0 \\
    c^{\{S\}}_2 & 0 & 0 & 0\\
    \frac{2}{3} & 0 & 0 & 0 \\
    1 & 0 & 0 & 0
    \end{bmatrix},\quad \Omega^{\{1\}} = \begin{bmatrix}
    0 & 0 & 0 & 0 \\
    0 & 0 & 0 & 0 \\
    -\frac{2}{3c^{\{S\}}_2} & \frac{2}{3c^{\{S\}}_2} & 0 & 0\\
    -\frac{3}{2} & 0 & \frac{3}{2} & 0
    \end{bmatrix}.
\end{align}
These satisfy all coupling conditions up to third order, and the corresponding slow base method satisfies all conditions up to order three.  Thus similarly to MERK2, we expect it to show third-order accuracy on problems with nonlinear $f^{\{F\}}$.  We again confirm this result on numerical tests in Section \ref{sec:numericalresults} using $c^{\{S\}}_2=1/2$.

Figure \ref{fig:imexmrisrmerk3explicitdoublejoint2} shows the stability regions for MERK3, $\mathcal{J}^{\{E\}}_{\alpha,10^2}$.  Notably, these regions show no degradation of stability as $\alpha$ is increased.

\subsection{MERK4}
\label{sec:merk4}

We may express MERK4 as an IMEX-MRI-SR method with 7 stages and $n_\Omega=3$.  The corresponding coefficients are provided in Appendix \ref{appendix:merk4}.

We find that this method satisfies all IMEX-MRI-SR coupling conditions through order four, and its slow base method satisfies all order conditions up through fourth order, so long as $c^{\{S\}}_6=(3-4c^{\{S\}}_5)/(4-6c^{\{S\}}_5)$. This restriction on $c^{\{S\}}_6$ is satisfied by the choices in \cite{luan_new_2020} of $c^{\{S\}}_2=1/2$, $c^{\{S\}}_3=1/2$, $c^{\{S\}}_4=1/3$, $c^{\{S\}}_5=5/6$, and $c^{\{S\}}_6=1/3$.  Thus like before, we expect this method to demonstrate  fourth-order accuracy for nonlinear $f^{\{F\}}$, which we confirm numerically in Section \ref{sec:numericalresults}.  Interestingly, the joint stability regions $\mathcal{J}^{\{E\}}_{\alpha,10^2}$ for MERK4 are empty.  However, given the reliability of this method in \cite{luan_new_2020} and our own results from Section \ref{sec:numericalresults}, we believe that these empty regions more strongly indicate a deficiency in these definitions of joint stability than any actual issues with MERK4 itself.

\subsection{MERK5}
\label{sec:merk5}

The IMEX-MRI-SR method that corresponds with MERK5 has 11 stages and $n_\Omega=4$. These coefficients are given in Appendix \ref{appendix:merk5}.  When analyzing this method, we find that when
\begin{align*}
    c^{\{S\}}_9=\frac{12-15c^{\{S\}}_{10}-15c^{\{S\}}_8+20 c^{\{S\}}_{10} c^{\{S\}}_8}{15-20c^{\{S\}}_{10}-20c^{\{S\}}_8+30c^{\{S\}}_{10}c^{\{S\}}_8}
\end{align*}
the method satisfies all IMEX-MRI-SR coupling conditions up through fourth-order, and its slow base method satisfies all order conditions up through fifth-order.  This condition is satisfied by the choice in \cite{luan_new_2020} of $c^{\{S\}}_2=1/2$, $c^{\{S\}}_3=1/2$, $c^{\{S\}}_4=1/3$, $c^{\{S\}}_5=1/2$, $c^{\{S\}}_6=1/3$, $c^{\{S\}}_7=1/4$, $c^{\{S\}}_8=7/10$, $c^{\{S\}}_9=1/2$, and $c^{\{S\}}_{10}=2/3$, therefore we expect it to have at least fourth-order accuracy without restriction on the linearity of $f^{\{F\}}$.  As we do not have fifth-order IMEX-MRI-SR order conditions, we cannot check these for MERK5; however, our numerical tests in Section \ref{sec:numericalresults} indeed show fifth-order convergence for nonlinear $f^{\{F\}}$.  As with MERK4, the MERK5 joint stability regions $\mathcal{J}^{\{E\}}_{\alpha,10^2}$ are empty.

\section{Numerical Results}
\label{sec:numericalresults}

In this section we examine the convergence rates for our newly-proposed IMEX-MRI-SR methods from Section \ref{sec:examplemethods}, as well as for MERK methods applied to problems with nonlinear $f^{\{F\}}$.  We also compare the efficiency of our IMEX-MRI-SR methods against both IMEX-MRI-GARK and legacy Lie-Trotter and Strang-Marchuk methods from \cite{chinomona_implicit-explicit_2021}.

\subsection{KPR}
\label{sec:kprproblemsec}

The Kv{\ae}rn{\o}-Prothero-Robinson problem is a coupled system of IVPs which has been widely used for testing multirate algorithms, since it is nonlinear, non-autonomous, includes stiffness and multirate tuning parameters, and has an analytical solution.  We use the same formulation and partitioning as in \cite{chinomona_implicit-explicit_2021},
\begin{equation}
  \label{eq:kprproblem}
  \begin{split}
    \begin{bmatrix} u' \\ v' \end{bmatrix} &= \Lambda
    \begin{bmatrix} \frac{-3+u^2-\cos(\beta t)}{2u} \\ \frac{-2+u^2-\cos(t)}{2v} \end{bmatrix}
    - \begin{bmatrix} \frac{\beta\sin(\beta t)}{2u} \\ \frac{\sin(t)}{2v} \end{bmatrix},
    \quad
    \begin{bmatrix} u(0) \\ v(0) \end{bmatrix} =
    \begin{bmatrix} 2 \\ \sqrt{3} \end{bmatrix} \\
    \Lambda &= \begin{bmatrix}
      \lambda^{\{F\}} & \frac{1-\varepsilon}{\alpha}(\lambda^{\{F\}}-\lambda^{\{S\}}) \\
      -\alpha\varepsilon(\lambda^{\{F\}}-\lambda^{\{S\}}) & \lambda^{\{S\}}
    \end{bmatrix}
  \end{split}
\end{equation}
for $t\in[0,{5\pi}/{2}]$, with parameters $\lambda^{\{F\}}=-10$, $\lambda^{\{S\}}=-1$, $\varepsilon=0.1$, $\alpha=1$, and $\beta=20$. This problem has solution
\begin{align}
    u(t) = \sqrt{3+\cos(\beta t)},\quad v(t) = \sqrt{2+\cos(t)}.
\end{align}
We partition the problem as
\begin{equation}
  \label{eq:kprpartitioning}
  \begin{split}
    f^{\{E\}}&=\begin{bmatrix} 0 \\ \frac{\sin(t)}{2v} \end{bmatrix},
    \quad f^{\{I\}}=\begin{bmatrix} 0 & 0 \\ 0 & 1 \end{bmatrix}
    \Lambda
    \begin{bmatrix} \frac{-3+u^2-\cos(\beta t)}{2u} \\ \frac{-2+u^2-\cos(t)}{2v} \end{bmatrix} \\
    f^{\{F\}}&=\begin{bmatrix} 1 & 0 \\ 0 & 0 \end{bmatrix} \Lambda
    \begin{bmatrix} \frac{-3+u^2-\cos(\beta t)}{2u} \\ \frac{-2+u^2-\cos(t)}{2v} \end{bmatrix}.
  \end{split}
\end{equation}

In Figure \ref{fig:kprconvergenceplots}, we plot the convergence as $H$ is refined for MERK methods and implicit-explicit methods, including all provided IMEX-MRI-SR methods from Section \ref{sec:examplemethods}, IMEX-MRI-GARK3(a,b), IMEX-MRI-GARK4, Lie-Trotter and Strang-Marchuk.  We see that all methods converge at their expected rates.  Notably, as expected from Section \ref{sec:merkasimexmrisr}, the MERK methods show no convergence issues even though $f^{\{F\}}$ in \eqref{eq:kprpartitioning} is nonlinear.

In these tests, we combined multirate methods with explicit inner Runge--Kutta methods of the same order, with the only exception of pairing a second-order inner solver with the first-order Lie-Trotter method.  Each of Lie-Trotter, Strang-Marchuk, and IMEX-MRI-SR2(1) used the second-order Heun method given by the Butcher table $\begin{array}{c|cc} 0 & 0 & 0 \\ 1 & 1 & 0 \\ \hline & 1/2 & 1/2 \end{array}$. The IMEX-MRI-SR3(2) and IMEX-MRI-GARK3(a,b) methods used the third-order method by Bogacki and Shampine \cite{bogacki_32_1989}. The IMEX-MRI-SR4(3) and IMEX-MRI-GARK4 methods used the fourth-order method by Zonneveld \cite{zonneveld1963automatic}.

We measured error at 10 equally-spaced points in the time interval.  The estimated convergence rates for each method, using a least squares fit of log(Max Error) versus log($H$), are in the legend parentheses. For the MERK convergence tests, we used $H=\pi/2^k$, $k=2,...,9$ and for the implicit-explicit method convergence tests, we used $H=\pi/2^k$, $k=4,...,11$. For all tests we used fast time step size $h=H/10$. The implicit-explicit methods used a standard Newton-Raphson method with a banded linear solver for the implicit solves.

\begin{figure}[htb]
\centering
\begin{subfigure}[b]{0.45\textwidth}
    \centering
    \includegraphics[width=\textwidth]{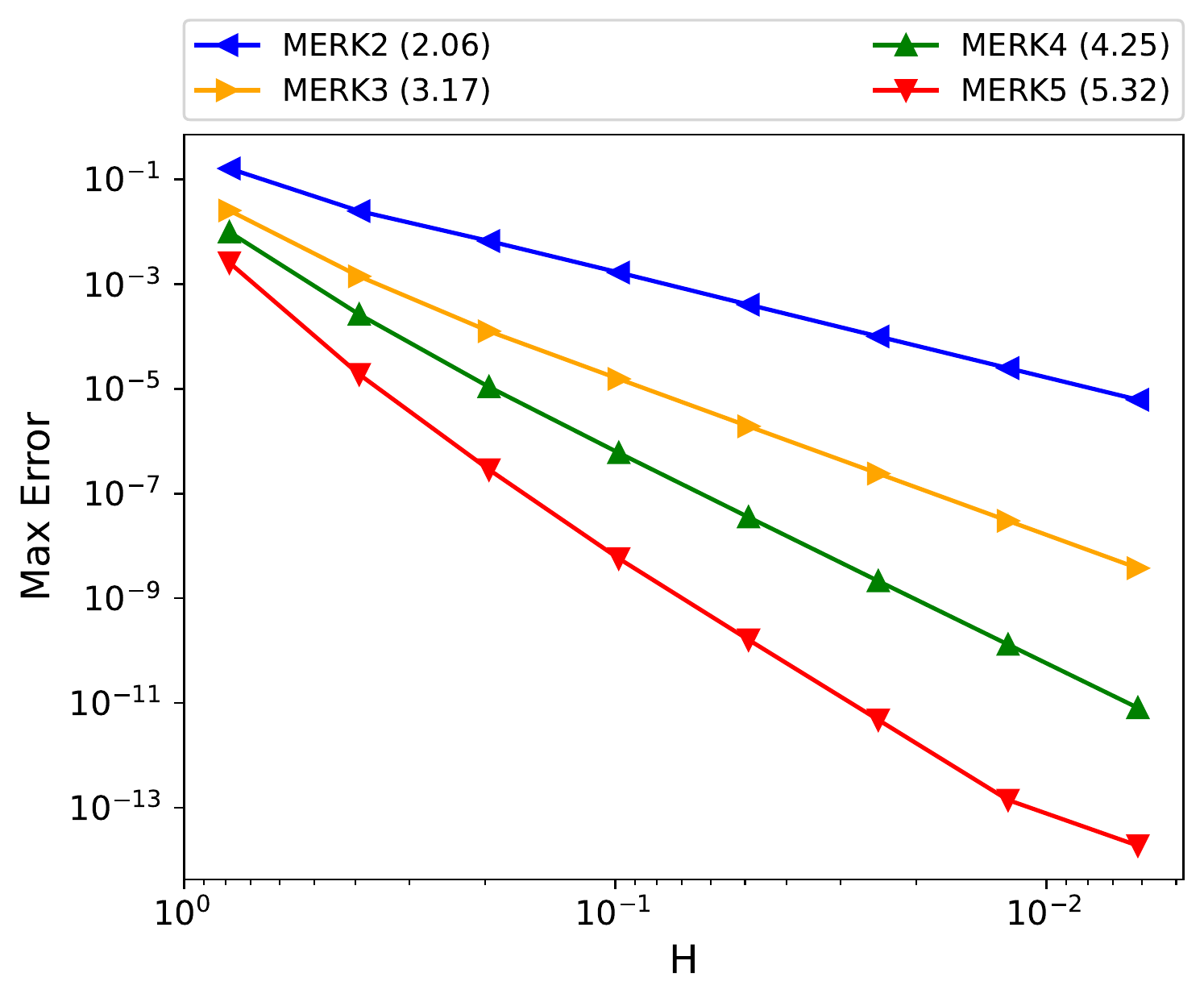}
\end{subfigure}
\hfill
\begin{subfigure}[b]{0.45\textwidth}
    \centering
    \includegraphics[width=\textwidth]{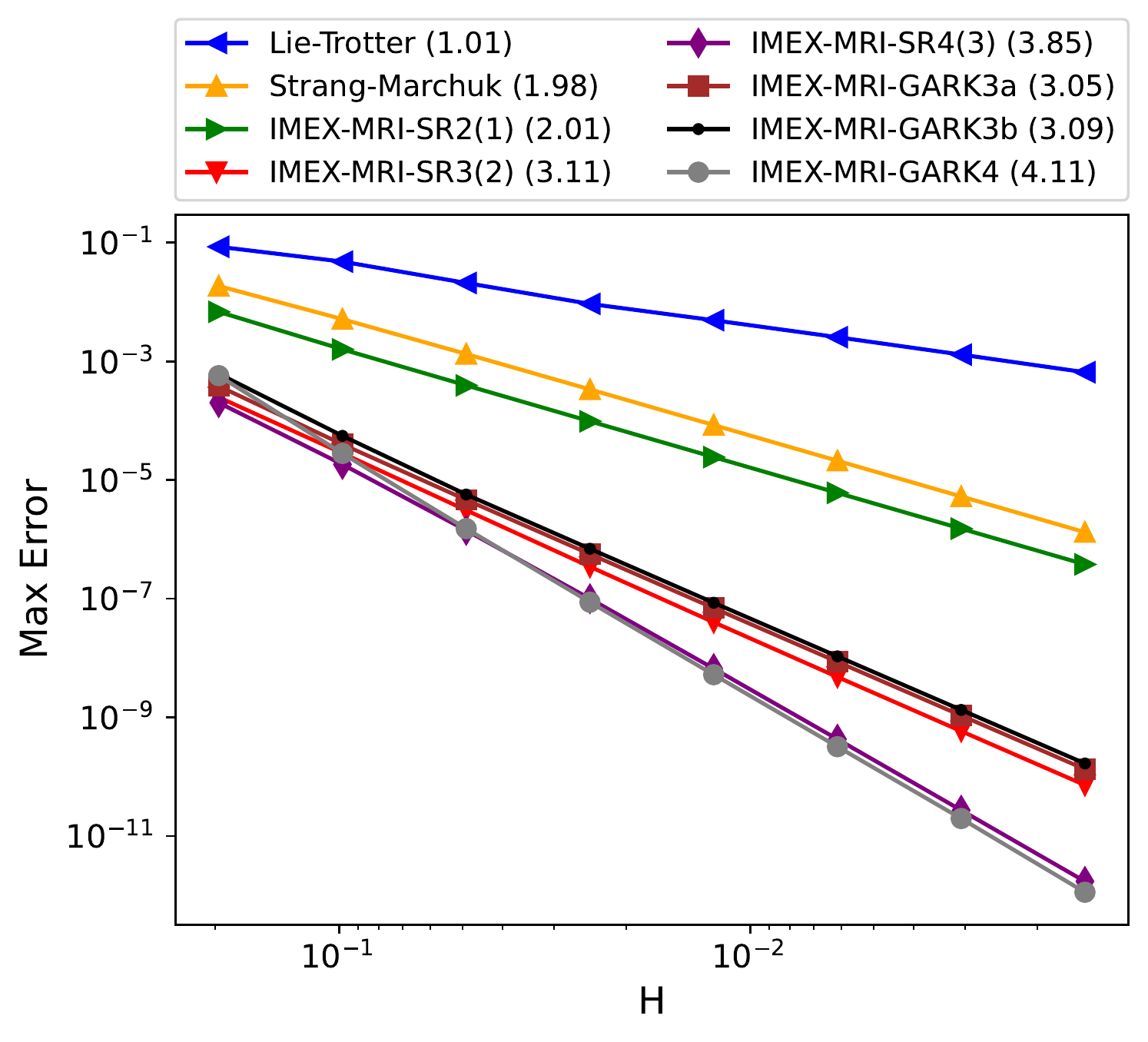}
\end{subfigure}
\caption{Convergence for the KPR test problem \eqref{eq:kprproblem} for MERK methods (left) and implicit-explicit methods (right) using the partitioning \eqref{eq:kprpartitioning}. All methods converge at the expected theoretical rates (with measured convergence rates in parentheses), including MERK methods using the given nonlinear fast partition.}
\label{fig:kprconvergenceplots}
\end{figure}

\subsection{Stiff Brusselator}
\label{sec:brusselatorproblemsec}

The stiff brusselator problem is an advection-reaction-diffusion system of nonlinear partial differential equations. It is a modification to the standard brusselator \cite{hairer1996stiff} used in \cite{chinomona_implicit-explicit_2021}, from which we use the same formulation and partitioning:
\begin{align*}
    u_t &= \alpha_u u_{xx} + \rho_u u_x + r_u(a - (w-1)u+u^2v), \\
    v_t &= \alpha_v v_{xx} + \rho_v v_x + r_v(wu-u^2v), \\
    w_t &= \alpha_w w_{xx} + \rho_w w_x + r_w(\frac{b-w}{\varepsilon} - wu),
\end{align*}
for $t\in[0,3]$ and $x\in[0,1]$, with stationary boundary conditions
\begin{align*}
    u_t(t,0) = u_t(t,1) = v_t(t,0) = v_t(t,1) = w_t(t,0) = w_t(t,1),
\end{align*}
and initial conditions
\begin{align*}
    u(0,x) &= a+0.1\sin(\pi x), \\
    v(0,x) &= b/a+0.1\sin(\pi x), \\
    w(0,x) &= b+0.1\sin(\pi x).
\end{align*}
We partition the problem as
\begin{align*}
  f^{\{I\}} = \begin{bmatrix} \alpha_u u_{xx} \\ \alpha_v v_{xx} \\ \alpha_w w_{xx} \end{bmatrix},\quad f^{\{E\}} = \begin{bmatrix} \rho_u u_x \\ \rho_v v_x \\ \rho_w w_x \end{bmatrix},\quad
  f^{\{F\}} = \begin{bmatrix} r_u(a - (w-1)u+u^2v) \\ r_v(wu-u^2v) \\ r_w(\frac{b-w}{\varepsilon} - wu) \end{bmatrix},
\end{align*}
and use second order centered difference approximations for all spatial derivative operators. This problem has no analytical solution, so we used MATLAB's \emph{ode15s} with tight tolerances $AbsTol=10^{-14}$ and $RelTol=2.5\times 10^{-14}$ to generate reference solutions.

\subsubsection{Fixed Time Step}
\label{sec:brusselatorfixedstepsec}
In this section, we compare runtime efficiency of the splitting, IMEX-MRI-SR, and IMEX-MRI-GARK methods using fixed time step sizes. We use the same fixed parameters $\alpha_u=\alpha_v=\alpha_w=10^{-2}$, $\rho_u=\rho_v=\rho_w=10^{-3}$, $r_u=r_v=r_w=1$, $a=0.6$, $b=2$, and $\varepsilon=10^{-2}$ with initial conditions
\begin{equation*}
u(0) = a + 0.1\sin(\pi x), \; v(0) = b/a + 0.1\sin(\pi x), \; w(0) = b + 0.1\sin(\pi x),
\end{equation*}
for 201 and 801 grid points as in \cite{chinomona_implicit-explicit_2021}. All methods used fast time steps of $h=H/10$ and all methods used the same inner methods and implicit algebraic solvers as in Section \ref{sec:kprproblemsec}.

\begin{figure}[h]
\centering
\includegraphics[width=0.9\textwidth]{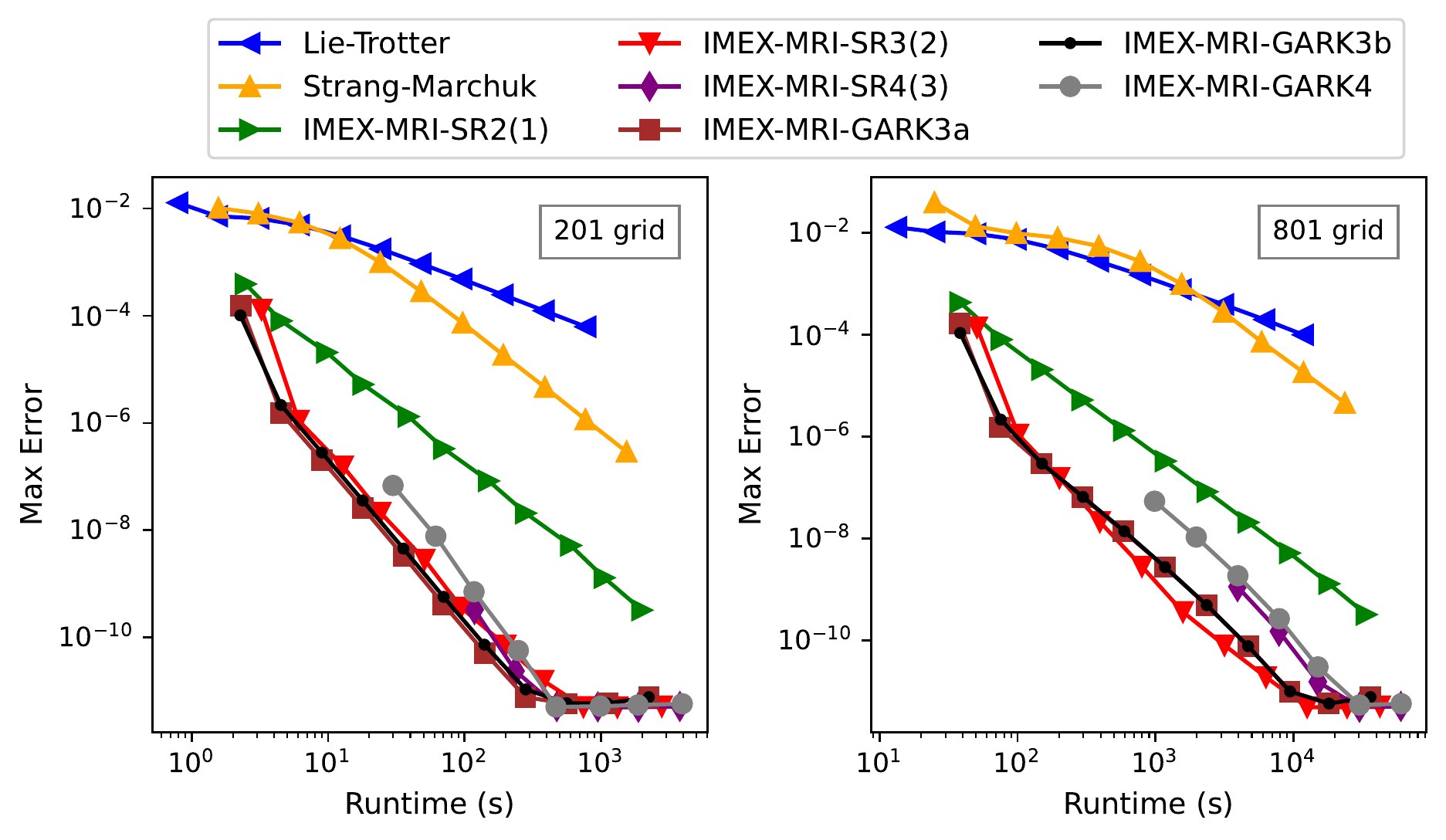}
\caption{Efficiency for stiff brusselator problem using 201 grid points (left) and 801 grid points (right). Estimated least-squares convergence rates before settling at the error-floor are (0.76,1.59,2.00,3.09,3.00,3.36,3.25,3.41) and (0.72,1.24,2.01,2.90,1.90,2.71,2.69,2.42) for the 201 and 801 grids, respectively, for Lie-Trotter, Strang-Marchuk, IMEX-MRI-SR2(1), IMEX-MRI-SR3(2), IMEX-MRI-SR4(3), IMEX-MRI-GARK3a, IMEX-MRI-GARK3b, IMEX-MRI-GARK4.}
\label{fig:brussruntimescombined}
\end{figure}

In Figure \ref{fig:brussruntimescombined} we plot the observed maximum solution error over ten equally spaced points in the time interval using step sizes of $H=0.1\cdot 2^{-k}$, $k=0,...,10$.  Both splitting methods and both fourth-order methods experience significant order reduction, with IMEX-MRI-SR4(3) taking the biggest hit in reducing an entire order of accuracy for the 201 grid, and two orders of accuracy for the 801 grid. The second- and third-order IMEX-MRI-SR and IMEX-MRI-GARK methods all achieve their expected order for the 201 grid. The third-order accurate methods experience only slight order reduction for the 801 grid while IMEX-MRI-SR2(1) remains steady at its expected order. All methods exhibit an error floor of approximately $10^{-11}$, that is likely caused by the accuracy of the reference solution.

The stiffness of this problem highlights the stability limitations of the fourth order methods, IMEX-MRI-SR4(3) and IMEX-MRI-GARK4, which was observed in \cite{chinomona_implicit-explicit_2021}. For the 201 grid, IMEX-MRI-SR4(3) and IMEX-MRI-GARK4 were unstable for step sizes greater than $1/320$ and $1/80$ respectively. For the 801 grid, IMEX-MRI-SR4(3) and IMEX-MRI-GARK4 were unstable for step sizes greater than $1/640$ and $1/160$, respectively.

We can see that IMEX-MRI-SR2(1) is \emph{far} more efficient than the first- and second-order splitting methods, providing errors two to three orders of magnitude smaller for the same runtimes. It also has a steady rate of error decrease as runtime increases, while the splitting methods show periods of stagnation at larger step sizes.

The third-order and fourth-order methods tend to have similar efficiency on this problem, attaining similar error for similar runtimes. The third-order IMEX-MRI-GARK methods have a slight edge for the 201 grid and for some error ranges in the 801 grid, but IMEX-MRI-SR3(2) becomes the most efficient method for the 801 grid for errors between approximately $10^{-7}$ and $10^{-11}$. The third-order methods are all more efficient than the fourth-order methods, where IMEX-MRI-SR4(3) maintains a slight but consistent edge over IMEX-MRI-GARK4.

\subsubsection{Adaptive Time Step}
\label{sec:brusselatoradaptivestepsec}

In this section, we compare work-precision efficiency for the IMEX-MRI-SR methods in the adaptive time step context. Because the splitting and IMEX-MRI-GARK methods do not have embeddings, we omit them from these tests. We use the Constant-Constant controller from \cite{fishcontroller2022} with the recommended parameters $k_1=0.42$, $k_2=0.44$, and the recommended fast error measurement strategy, LASA-mean. This controller adapts $H$ and $M$, such that the inner time step size $h=H/M$, in a similar manner to a standard I-controller.

We use time-varying parameters $\alpha_u=\alpha_v=\alpha_w=\rho_u=\rho_v=\rho_w=6\times10^{-5}+5\times10^{-5}\cos(\pi t)$, $r_u=r_v=r_w=0.6+0.5\cos(4\pi t)$ adapted from \cite{fishcontroller2022} with the same fixed parameters $a=1$, $b=3.5$, $\varepsilon=10^{-3}$ and initial conditions
\begin{equation*}
u(0) = 1.2 + 0.1\sin(\pi x), \; v(0) = 3.1 + 0.1\sin(\pi x), \; w(0) = 3 + 0.1\sin(\pi x),
\end{equation*} with
101 grid points. All IMEX-MRI-SR methods used the same inner methods and implicit algebraic solvers as in Section \ref{sec:kprproblemsec}.

\begin{figure}[h]
\centering
\includegraphics[width=0.85\textwidth]{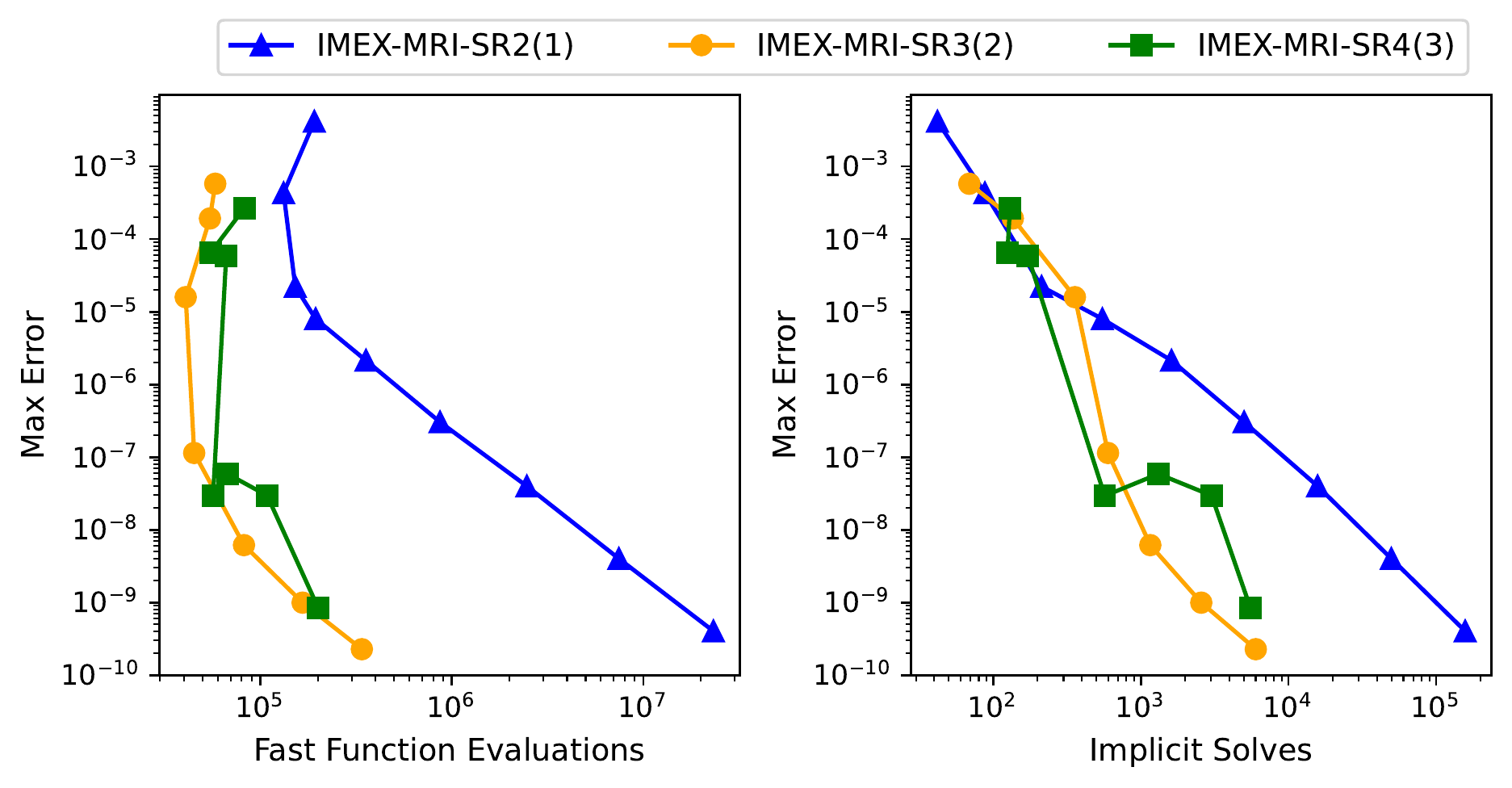}
\caption{Fast function evaluations (left) and total implicit solves (right) versus the observed maximum error for the stiff brusselator problem.}
\label{fig:brussworkprecision}
\end{figure}

In Figure \ref{fig:brussworkprecision} we plot the number of fast function evaluations and the number of implicit solves, good indicators of overall cost at the fast and slow timescales, versus the observed maximum solution error over ten equally spaced points in the time interval when running with controller tolerance values of $\tol=10^{-k}$, $k=1,...,9$. We use $\tol^{\{S\}}=\tol^{\{F\}}=\frac{1}{2}\tol$ in our tests for simplicity. We note that IMEX-MRI-SR3(2) failed the tests with $\tol=10^{-3},10^{-5}$ and IMEX-MRI-SR4(3) failed the tests with $\tol=10^{-4},10^{-5}$ due to getting stuck in oscillations between successful and failed steps.

We can see that IMEX-MRI-SR2(1) is much less efficient than the higher order methods in terms of fast function evaluations (farthest from the bottom-left corner), where IMEX-MRI-SR3(2) is generally the most efficient (closest to the bottom-left corner) across the range of errors. IMEX-MRI-SR4(3) is generally comparable to IMEX-MRI-SR3(2) in terms of fast function evaluations but occasionally gets ``stuck,'' providing approximately the same error value for different total fast function evaluations, depending on the value of $\tol$. This is likely an indication that IMEX-MRI-SR4(3) has a low quality embedding which provides inaccurate error estimates, possibly stemming from the embedding coefficients being defined by minimizing $\|\hat{\tau}^{(4)}\|_2$ rather than the C-statistic.

The total number of implicit solves is comparable across all three methods for errors larger than approximately $10^{-5}$. Below that, IMEX-MRI-SR3(2) and IMEX-MRI-SR4(3) are again comparable with IMEX-MRI-SR2(1) falling further behind. We see the same phenomenon of IMEX-MRI-SR4(3) getting ``stuck'' at certain error values, achieving the same error for varying total implicit solves.



\section{Conclusions}
\label{sec:conclusions}

We introduce a new class of multirate time integration methods which builds off of previous work in  IMEX-MRI-GARK and MERK methods, serving to improve various aspects of each.  The proposed class of IMEX-MRI-SR methods are flexible, allowing IMEX treatment of the slow time scale while allowing the use of any viable IVP solver for the fast time scale.  These methods remove the sorted abscissae requirement of IMEX-MRI-GARK methods since they start each internal stage at the beginning of the time step, thereby dramatically simplifying their order conditions and allowing introduction of embeddings.  IMEX-MRI-SR methods can also be viewed as an extension to MERK methods, in that these allow implicitness at the slow time scale and nonlinearity at the fast time scale.

The convergence theory of IMEX-MRI-SR methods leverages GARK theory \cite{Sandu2015}, through which we established order conditions for methods of orders one through four.  Due to their structural similarity to IMEX-MRI-SR methods, we leveraged these new order conditions to analyze the previously-proposed MERK methods without their restriction to linear fast partitions.  Using this analytical framework, we provided the first theoretical justification that these MERK methods (at least up through fourth order) should retain their high orders of accuracy even on problems with nonlinear fast partitions.

We analyzed joint stability as in \cite{chinomona_implicit-explicit_2021}, as well as in simplified implicit- or explicit-only senses of joint stability. With these reduced definitions of stability, we gain some insight into what happens when joint stability regions break down.

Using this theoretical framework, we provided three new IMEX-MRI-SR methods. These included a second-order method with a first-order embedding, and a third-order method with a second-order embedding, both derived from scratch by solving all order conditions simultaneously.  We additionally provided a fourth-order method with a third-order embedding, that we derived from the existing LIRK4 ARK method \cite{calvo2001linearly}.

We experimentally examined convergence for the three new methods and the existing MERK methods on the KPR test problem, finding that they all converge at their expected orders of accuracy. We also provided experimental efficiency results on the stiff brusselator problem, a nonlinear advection-reaction-diffusion system of PDEs, where we found that our methods are competitive with IMEX-MRI-GARK methods and even surpass them in some cases.  We also found that IMEX-MRI-SR2(1) is vastly more efficient than the similarly second-order Strang-Marchuk operator splitting method. We found that in the context of adaptive time stepping, the third- and fourth-order IMEX-MRI-SR methods were comparable in work required for achieving a given error, while the second-order method lagged behind.

More work remains to be done on IMEX-MRI-SR methods. Higher order conditions can be derived which, while tedious due to the number of conditions at higher orders, is tractable due to the use of Kronecker product identities.  Further analysis should be done on the factors that most strongly affect joint stability and what conditions, if any, can be enforced to ensure a non-empty joint stability region. Additionally, more methods should be derived, in particular an embedded fourth-order method with improved joint stability over the method provided here.

\section{Acknowledgments}
\label{sec:acknowledgments}

The authors would like to thank Vu Thai Luan for his help in gaining deeper understanding of MERK methods. We would also like to thank the SMU Center for Research Computing for use of the Maneframe2 computing cluster, where we performed all simulations
reported in this work


\bibliographystyle{siamplain}
\bibliography{references}

\newpage
\appendix

\section{IMEX-MRI-SR2(1) Coefficients}
\label{appendix:imexmrisr21}

\begin{align}
    c^{\{S\}}=\begin{bmatrix}
    0 \\ \frac35 \\ \frac{4}{15} \\ 1
    \end{bmatrix},\ \Omega^{\{0\}} = \begin{bmatrix}
    0 & 0 & 0 & 0 \\ \frac{3}{5} & 0 & 0 & 0 \\ \frac{14}{165} & \frac{2}{11} & 0 & 0 \\ -\frac{13}{54} & \frac{137}{270} & \frac{11}{15} & 0 \\ \hline -\frac14 & \frac12 & \frac{3}{4} & 0
    \end{bmatrix},\ \Gamma=\begin{bmatrix}
    0 & 0 & 0 & 0 \\ -\frac{11}{23} & \frac{11}{23} & 0 & 0 \\ -\frac{6692}{52371} & -\frac{18355}{52371} & \frac{11}{23} & 0 \\ \frac{11621}{90666} & -\frac{215249}{226665} & \frac{17287}{50370} & \frac{11}{23} \\ \hline -\frac{31}{12} & -\frac{1}{6} & \frac{11}{4} & 0
    \end{bmatrix}.
\end{align}

\section{IMEX-MRI-SR3(2) Coefficients}
\label{appendix:imexmrisr32}

\begin{equation}
\begin{split}
    c^{\{S\}}&=\begin{bmatrix}
    0 \\ \frac{23}{34} \\ \frac{4}{5} \\ \frac{17}{15} \\ 1
    \end{bmatrix},\ \Omega^{\{0\}} = \begin{bmatrix}
    0 & 0 & 0 & 0 & 0\\ \frac{23}{34} & 0 & 0 & 0 & 0 \\ \frac{71}{70} & -\frac{3}{14} & 0 & 0 & 0 \\ \frac{124}{1155} & \frac{4}{7} & \frac{5}{11} & 0 & 0 \\ \frac{162181}{187680} & \frac{119}{1380} & \frac{11}{32} & -\frac{5}{17} & 0 \\ \hline \frac{76355}{74834} & -\frac{46}{31} & \frac{67}{34} & -\frac{36}{71} & 0
    \end{bmatrix},\\ \Omega^{\{1\}}&=\begin{bmatrix}
    0 & 0 & 0 & 0 & 0 \\ 0 & 0 & 0 & 0 & 0 \\ -\frac{14453}{63825} & \frac{14453}{63825} & 0 & 0 & 0 \\ -\frac{2101267877}{1206582300} & -\frac{2476735438}{301645575} & -\frac{13575085}{2098404} & 0 & 0 \\ -\frac{762580446799}{588660102960} & \frac{11083240219}{4328383110} & -\frac{211274129}{100368304} & \frac{89562055}{106641323} & 0 \\ \hline -\frac{3732974}{2278035} & \frac{13857574}{2278035} & -\frac{52}{9} & \frac{4}{3} & 0
    \end{bmatrix} \\
    \Gamma &= \begin{bmatrix}
    0 & 0 & 0 & 0 & 0 \\ -\frac47 & \frac47 & 0 & 0 & 0 \\ -\frac{2707004}{3127425} & \frac{919904}{3127425} & \frac47 & 0 & 0 \\ \frac{852879271}{703839675} & -\frac{1575000496}{703839675} & \frac{5}{11} & \frac47 & 0 \\ \frac{43136869}{2019912118} & -\frac{73810600}{1009956059} & -\frac{17653551}{87822266} & -\frac{13993902}{43911133} & \frac47 \\ \hline -\frac{179}{4140} & \frac{799}{14490} & \frac{1}{14} & -\frac{1}{12} & 0
    \end{bmatrix}
\end{split}
\end{equation}

\section{IMEX-MRI-SR4(3) Coefficients}
\label{appendix:imexmrisr43}

\begin{equation}
\begin{split}
    c^{\{S\}}&=\begin{bmatrix}
    0 \\ \frac14 \\ \frac34 \\ \frac{11}{20} \\ \frac12 \\ 1 \\ 1
    \end{bmatrix},\ \Omega^{\{0\}} = \begin{bmatrix}
    0 & 0 & 0 & 0 & 0 & 0 & 0 \\
    \frac14 & 0 & 0 & 0 & 0 & 0 & 0 \\
    \frac98 & -\frac38 & 0 & 0 & 0 & 0 & 0 \\
    \frac{187}{2340} & \frac79 & -\frac{4}{13} & 0 & 0 & 0 & 0 \\
    \frac{64}{165} & \frac16 & -\frac35 & \frac{6}{11} & 0 & 0 & 0 \\
    \frac{1816283}{549120} & -\frac29 & -\frac{4}{11} & -\frac16 & -\frac{2561809}{1647360} & 0 & 0 \\
    0 & \frac{7}{11} & -\frac{2203}{264} & \frac{10825}{792} & -\frac{85}{12} & \frac{841}{396} & 0 \\ \hline
    \frac{1}{400} & \frac{49}{12} & \frac{43}{6} & -\frac{7}{10} & -\frac{85}{12} & -\frac{2963}{1200} & 0
    \end{bmatrix} \\
    \Omega^{\{1\}} &= \begin{bmatrix}
    0 & 0 & 0 & 0 & 0 & 0 & 0 \\
    0 & 0 & 0 & 0 & 0 & 0 & 0 \\
    -\frac{11}{4} & \frac{11}{4} & 0 & 0 & 0 & 0 & 0 \\
    -\frac{1228}{2925} & -\frac{92}{225} & \frac{808}{975} & 0 & 0 & 0 & 0 \\
    -\frac{2572}{2805} & \frac{167}{255} & \frac{199}{136} & -\frac{1797}{1496} & 0 & 0 & 0 \\
    -\frac{1816283}{274560} & \frac{253}{36} & -\frac{23}{44} & \frac{76}{3} & -\frac{20775791}{823680} & 0 & 0 \\
    0 & \frac{107}{132} & \frac{1289}{88} & -\frac{9275}{792} & 0 & -\frac{371}{99} & 0 \\ \hline
    -\frac{1}{200} & -\frac{137}{24} & -\frac{235}{16} & \frac{1237}{80} & 0 & \frac{2963}{600} & 0
    \end{bmatrix} \\
    \Gamma &= \begin{bmatrix}
    0 & 0 & 0 & 0 & 0 & 0 & 0 \\
    -\frac{1}{4} & \frac{1}{4} & 0 & 0 & 0 & 0 & 0 \\
    \frac14 & -\frac12 & \frac14 & 0 & 0 & 0 & 0 \\
    \frac{13}{100} & -\frac{7}{30} & -\frac{11}{75} & \frac14 & 0 & 0 & 0 \\
    \frac{6}{85} & -\frac{301}{1360} & -\frac{99}{544} & \frac{45}{544} & \frac14 & 0 & 0 \\
    0 & -\frac94 & -\frac{19}{48} & -\frac{75}{16} & \frac{85}{12} & \frac14 & 0 \\
    0 & 0 & 0 & 0 & 0 & 0 & 0 \\ \hline
    0 & 0 & 0 & 0 & 0 & 0 & 0
    \end{bmatrix}
\end{split}
\end{equation}

\section{MERK4 IMEX-MRI-SR Coefficients}
\label{appendix:merk4}
We list the non-zero coefficients of the MERK5 method's IMEX-MRI-SR formulation below.
\begin{equation}
\begin{split}
    \Omega^{\{0\}}_{i,1} &= c^{\{S\}}_i,\ i=1,...,s^{\{S\}}
\end{split}
\end{equation}

\begin{equation}
\begin{split}
    \Omega^{\{1\}}_{3,1} &= -\Omega^{\{1\}}_{3,2} ,\  \Omega^{\{1\}}_{3,2} = \frac{c^{\{S\}\times 2}_3}{c^{\{S\}}_2},\  \Omega^{\{1\}}_{4,1} = -\Omega^{\{1\}}_{4,2} ,\  \Omega^{\{1\}}_{4,2} = \frac{c^{\{S\}\times 2}_4}{c^{\{S\}}_2},\\
    \Omega^{\{1\}}_{5,1} &= -(\Omega^{\{1\}}_{5,3}+\Omega^{\{1\}}_{5,4}),\  \Omega^{\{1\}}_{5,3} = -\frac{c^{\{S\}}_4c^{\{S\}\times 2}_5}{c^{\{S\}}_3(c^{\{S\}}_3-c^{\{S\}}_4)},\  \Omega^{\{1\}}_{5,4} =\frac{c^{\{S\}}_3c^{\{S\}\times 2}_5}{c^{\{S\}}_4(c^{\{S\}}_3-c^{\{S\}}_4)} \\
    \Omega^{\{1\}}_{6,1} &= -(\Omega^{\{1\}}_{6,3}+\Omega^{\{1\}}_{6,4}),\  \Omega^{\{1\}}_{6,3} = -\frac{c^{\{S\}}_4c^{\{S\}\times 2}_6}{c^{\{S\}}_3(c^{\{S\}}_3-c^{\{S\}}_4)},\  \Omega^{\{1\}}_{6,4} =\frac{c^{\{S\}}_3c^{\{S\}\times 2}_6}{c^{\{S\}}_4(c^{\{S\}}_3-c^{\{S\}}_4)} \\
    \Omega^{\{1\}}_{7,1} &= -(\Omega^{\{1\}}_{7,5}+\Omega^{\{1\}}_{7,6}),\  \Omega^{\{1\}}_{7,5} = -\frac{c^{\{S\}}_6}{c^{\{S\}}_5(c^{\{S\}}_5-c^{\{S\}}_6)},\  \Omega^{\{1\}}_{7,6} =\frac{c^{\{S\}}_5}{c^{\{S\}}_6(c^{\{S\}}_5-c^{\{S\}}_6)}
\end{split}
\end{equation}

\begin{equation}
\begin{split}
    \Omega^{\{1\}}_{5,1} &= -(\Omega^{\{1\}}_{5,3}+\Omega^{\{1\}}_{5,4}),\  \Omega^{\{1\}}_{5,3} = \frac{1}{c^{\{S\}}_3(c^{\{S\}}_3-c^{\{S\}}_4)},\  \Omega^{\{1\}}_{5,4} =-\frac{1}{c^{\{S\}}_4(c^{\{S\}}_3-c^{\{S\}}_4)} \\
    \Omega^{\{1\}}_{6,1} &= -(\Omega^{\{1\}}_{6,3}+\Omega^{\{1\}}_{6,4}),\  \Omega^{\{1\}}_{6,3} = \frac{1}{c^{\{S\}}_3(c^{\{S\}}_3-c^{\{S\}}_4)},\  \Omega^{\{1\}}_{6,4} =-\frac{1}{c^{\{S\}}_4(c^{\{S\}}_3-c^{\{S\}}_4)}  \\
    \Omega^{\{1\}}_{7,1} &= -(\Omega^{\{1\}}_{7,5}+\Omega^{\{1\}}_{7,6}),\  \Omega^{\{1\}}_{7,5} = \frac{1}{c^{\{S\}}_5(c^{\{S\}}_5-c^{\{S\}}_6)},\  \Omega^{\{1\}}_{7,6} =-\frac{1}{c^{\{S\}}_6(c^{\{S\}}_5-c^{\{S\}}_6)}
\end{split}
\end{equation}

This method attains general fourth-order accuracy when \begin{align}
    c^{\{S\}}_6=\frac{3-4c^{\{S\}}_5}{4-6c^{\{S\}}_5}.
\end{align}

\newpage
\section{MERK5 IMEX-MRI-SR Coefficients}
\label{appendix:merk5}

We list the non-zero coefficients of the MERK5 method's IMEX-MRI-SR formulation below.
\begin{equation}
    \Omega^{\{0\}}_{i,1} = c^{\{S\}}_i,\ i=1,...,s^{\{S\}}
\end{equation}

\begin{equation}
\begin{split}
    \Omega^{\{1\}}_{3,1} &= -\Omega^{\{1\}}_{3,2} ,\  \Omega^{\{1\}}_{3,2} = c^{\{S\}\times 2}_3\alpha_2,\\
    \Omega^{\{1\}}_{4,1} &= -\Omega^{\{1\}}_{4,2} ,\  \Omega^{\{1\}}_{4,2} = c^{\{S\}\times 2}_4\alpha_2,\\
    \Omega^{\{1\}}_{5,1} &= -\left(\Omega^{\{1\}}_{5,3}+\Omega^{\{1\}}_{5,4}\right),\  \Omega^{\{1\}}_{5,3} = c^{\{S\}\times 2}_5\alpha_3,\  \Omega^{\{1\}}_{5,4} =c^{\{S\}\times 2}_5\alpha_4 \\
    \Omega^{\{1\}}_{6,1} &= -\left(\Omega^{\{1\}}_{6,3}+\Omega^{\{1\}}_{6,4}\right),\  \Omega^{\{1\}}_{6,3} = c^{\{S\}\times 2}_6\alpha_3,\  \Omega^{\{1\}}_{6,4} =c^{\{S\}\times 2}_6\alpha_4 \\
    \Omega^{\{1\}}_{7,1} &= -\left(\Omega^{\{1\}}_{7,3}+\Omega^{\{1\}}_{7,4}\right),\  \Omega^{\{1\}}_{7,3} = c^{\{S\}\times 2}_7\alpha_3,\  \Omega^{\{1\}}_{7,4} =c^{\{S\}\times 2}_7\alpha_4 \\
    \Omega^{\{1\}}_{8,1} &= -\left(\Omega^{\{1\}}_{8,5}+\Omega^{\{1\}}_{8,6}+\Omega^{\{1\}}_{8,7}\right),\ \Omega^{\{1\}}_{8,5} = c^{\{S\}\times 2}_8\alpha_5,\ \Omega^{\{1\}}_{8,6} = c^{\{S\}\times 2}_8\alpha_6,\ \Omega^{\{1\}}_{8,7} = c^{\{S\}\times 2}_8\alpha_7 \\
    \Omega^{\{1\}}_{9,1} &= -\left(\Omega^{\{1\}}_{9,5}+\Omega^{\{1\}}_{9,6}+\Omega^{\{1\}}_{9,7}\right),\ \Omega^{\{1\}}_{9,5} = c^{\{S\}\times 2}_9\alpha_5,\ \Omega^{\{1\}}_{9,6} = c^{\{S\}\times 2}_9\alpha_6,\ \Omega^{\{1\}}_{9,7} = c^{\{S\}\times 2}_9\alpha_7 \\
    \Omega^{\{1\}}_{10,1} &= -\left(\Omega^{\{1\}}_{10,5}+\Omega^{\{1\}}_{10,6}+\Omega^{\{1\}}_{10,7}\right),\ \Omega^{\{1\}}_{10,5} = c^{\{S\}\times 2}_{10}\alpha_5,\ \Omega^{\{1\}}_{10,6} = c^{\{S\}\times 2}_{10}\alpha_6,\ \Omega^{\{1\}}_{10,7} = c^{\{S\}\times 2}_{10}\alpha_7 \\
     \Omega^{\{1\}}_{11,1} &= -\left(\Omega^{\{1\}}_{11,8}+\Omega^{\{1\}}_{11,9}+\Omega^{\{1\}}_{11,10}\right),\ \Omega^{\{1\}}_{11,8} = \alpha_8,\ \Omega^{\{1\}}_{11,9} = \alpha_9,\ \Omega^{\{1\}}_{11,10} = \alpha_{10}
\end{split}
\end{equation}
where
\begin{equation}
\begin{split}
    \alpha_2 &= \frac{1}{c^{\{S\}}_2},\ \alpha_3 =\frac{c^{\{S\}}_3}{c^{\{S\}}_3(c^{\{S\}}_4-c^{\{S\}}_3)},\ \alpha_4=\frac{c^{\{S\}}_4}{c^{\{S\}}_4(c^{\{S\}}_3-c^{\{S\}}_4)} \\
    \alpha_5 &= \frac{c^{\{S\}}_6c^{\{S\}}_7}{c^{\{S\}}_5(c^{\{S\}}_5-c^{\{S\}}_6)(c^{\{S\}}_5-c^{\{S\}}_7)},\ \alpha_6 = \frac{c^{\{S\}}_5c^{\{S\}}_7}{c^{\{S\}}_6(c^{\{S\}}_6-c^{\{S\}}_5)(c^{\{S\}}_6-c^{\{S\}}_7)},\\
    \alpha_7 &= \frac{c^{\{S\}}_5c^{\{S\}}_6}{c^{\{S\}}_7(c^{\{S\}}_7-c^{\{S\}}_5)(c^{\{S\}}_7-c^{\{S\}}_6)},\ \alpha_8=\frac{c^{\{S\}}_9c^{\{S\}}_{10}}{c^{\{S\}}_8(c^{\{S\}}_8-c^{\{S\}}_9)(c^{\{S\}}_8-c^{\{S\}}_{10})} \\
    \alpha_9&=\frac{c^{\{S\}}_8c^{\{S\}}_{10}}{c^{\{S\}}_9(c^{\{S\}}_9-c^{\{S\}}_8)(c^{\{S\}}_9-c^{\{S\}}_{10})},\ \alpha_{10}=\frac{c^{\{S\}}_8c^{\{S\}}_9}{c^{\{S\}}_{10}(c^{\{S\}}_{10}-c^{\{S\}}_8)(c^{\{S\}}_{10}-c^{\{S\}}_{9})}
\end{split}
\end{equation}

\begin{equation}
\begin{split}
    \Omega^{\{2\}}_{5,1} &= -\left(\Omega^{\{2\}}_{5,3}+\Omega^{\{2\}}_{5,4}\right),\  \Omega^{\{2\}}_{5,3} = c^{\{S\}\times 2}_5\beta_3,\  \Omega^{\{2\}}_{5,4} =-c^{\{S\}\times 2}_5\beta_4 \\
    \Omega^{\{2\}}_{6,1} &= -\left(\Omega^{\{2\}}_{6,3}+\Omega^{\{2\}}_{6,4}\right),\  \Omega^{\{2\}}_{6,3} = c^{\{S\}\times 2}_6\beta_3,\  \Omega^{\{2\}}_{6,4} =-c^{\{S\}\times 2}_6\beta_4 \\
    \Omega^{\{2\}}_{7,1} &= -\left(\Omega^{\{2\}}_{7,3}+\Omega^{\{2\}}_{7,4}\right),\  \Omega^{\{2\}}_{7,3} = c^{\{S\}\times 2}_7\beta_3,\  \Omega^{\{2\}}_{7,4} =-c^{\{S\}\times 2}_7\beta_4 \\
    \Omega^{\{2\}}_{8,1} &= -\left(\Omega^{\{2\}}_{8,5}+\Omega^{\{2\}}_{8,6}+\Omega^{\{2\}}_{8,7}\right),\ \Omega^{\{2\}}_{8,5} = -c^{\{S\}\times 2}_8\beta_5,\ \Omega^{\{2\}}_{8,6} = -c^{\{S\}\times 2}_8\beta_6,\ \Omega^{\{2\}}_{8,7} = -c^{\{S\}\times 2}_8\beta_7 \\
    \Omega^{\{2\}}_{9,1} &= -\left(\Omega^{\{2\}}_{9,5}+\Omega^{\{2\}}_{9,6}+\Omega^{\{2\}}_{9,7}\right),\ \Omega^{\{2\}}_{9,5} = -c^{\{S\}\times 2}_9\beta_5,\ \Omega^{\{2\}}_{9,6} = -c^{\{S\}\times 2}_9\beta_6,\ \Omega^{\{2\}}_{9,7} = -c^{\{S\}\times 2}_9\beta_7 \\
    \Omega^{\{2\}}_{10,1} &= -\left(\Omega^{\{2\}}_{10,5}+\Omega^{\{2\}}_{10,6}+\Omega^{\{2\}}_{10,7}\right),\ \Omega^{\{2\}}_{10,5} = -c^{\{S\}\times 2}_{10}\beta_5,\ \Omega^{\{2\}}_{10,6} = -c^{\{S\}\times 2}_{10}\beta_6,\ \Omega^{\{2\}}_{10,7} = -c^{\{S\}\times 2}_{10}\beta_7 \\
     \Omega^{\{2\}}_{11,1} &= -\left(\Omega^{\{2\}}_{11,8}+\Omega^{\{2\}}_{11,9}+\Omega^{\{2\}}_{11,10}\right),\ \Omega^{\{2\}}_{11,8} = \beta_8,\ \Omega^{\{2\}}_{11,9} = \beta_9,\ \Omega^{\{2\}}_{11,10} = \beta_{10}
\end{split}
\end{equation}
where
\begin{equation}
\begin{split}
    \beta_3&=\frac{1}{c^{\{S\}}_3(c^{\{S\}}_4-c^{\{S\}}_3)},\ \beta_4=\frac{1}{c^{\{S\}}_4(c^{\{S\}}_3-c^{\{S\}}_4)} \\
    \beta_5 &= \frac{c^{\{S\}}_6+c^{\{S\}}_7}{c^{\{S\}}_5(c^{\{S\}}_5-c^{\{S\}}_6)(c^{\{S\}}_5-c^{\{S\}}_7)},\ \beta_6 = \frac{c^{\{S\}}_5+c^{\{S\}}_7}{c^{\{S\}}_6(c^{\{S\}}_6-c^{\{S\}}_5)(c^{\{S\}}_6-c^{\{S\}}_7)},\\
    \beta_7 &= \frac{c^{\{S\}}_5+c^{\{S\}}_6}{c^{\{S\}}_7(c^{\{S\}}_7-c^{\{S\}}_5)(c^{\{S\}}_7-c^{\{S\}}_6)},\ \beta_8=\frac{c^{\{S\}}_9+c^{\{S\}}_{10}}{c^{\{S\}}_8(c^{\{S\}}_8-c^{\{S\}}_9)(c^{\{S\}}_8-c^{\{S\}}_{10})} \\
    \beta_9&=\frac{c^{\{S\}}_8+c^{\{S\}}_{10}}{c^{\{S\}}_9(c^{\{S\}}_9-c^{\{S\}}_8)(c^{\{S\}}_9-c^{\{S\}}_{10})},\ \beta_{10}=\frac{c^{\{S\}}_8+c^{\{S\}}_9}{c^{\{S\}}_{10}(c^{\{S\}}_{10}-c^{\{S\}}_8)(c^{\{S\}}_{10}-c^{\{S\}}_{9})}
\end{split}
\end{equation}

\begin{equation}
\begin{split}
    \Omega^{\{3\}}_{8,1} &= -\left(\Omega^{\{3\}}_{8,5}+\Omega^{\{3\}}_{8,6}+\Omega^{\{3\}}_{8,7}\right),\ \Omega^{\{3\}}_{8,5} = c^{\{S\}\times 2}_8\gamma_5,\ \Omega^{\{3\}}_{8,6} = c^{\{S\}\times 2}_8\gamma_6,\ \Omega^{\{3\}}_{8,7} = c^{\{S\}\times 2}_8\gamma_7 \\
    \Omega^{\{3\}}_{9,1} &= -\left(\Omega^{\{3\}}_{9,5}+\Omega^{\{3\}}_{9,6}+\Omega^{\{3\}}_{9,7}\right),\ \Omega^{\{3\}}_{9,5} = c^{\{S\}\times 2}_9\gamma_5,\ \Omega^{\{3\}}_{9,6} = c^{\{S\}\times 2}_9\gamma_6,\ \Omega^{\{3\}}_{9,7} = c^{\{S\}\times 2}_9\gamma_7 \\
    \Omega^{\{3\}}_{10,1} &= -\left(\Omega^{\{3\}}_{10,5}+\Omega^{\{3\}}_{10,6}+\Omega^{\{3\}}_{10,7}\right),\ \Omega^{\{3\}}_{10,5} = c^{\{S\}\times 2}_{10}\gamma_5,\ \Omega^{\{3\}}_{10,6} = c^{\{S\}\times 2}_{10}\gamma_6,\ \Omega^{\{3\}}_{10,7} = c^{\{S\}\times 2}_{10}\gamma_7 \\
     \Omega^{\{3\}}_{11,1} &= -\left(\Omega^{\{3\}}_{11,8}+\Omega^{\{3\}}_{11,9}+\Omega^{\{3\}}_{11,10}\right),\ \Omega^{\{3\}}_{11,8} = \gamma_8,\ \Omega^{\{3\}}_{11,9} = \gamma_9,\ \Omega^{\{3\}}_{11,10} = \gamma_{10}
\end{split}
\end{equation}
where
\begin{equation}
\begin{split}
    \gamma_5 &= \frac{1}{c^{\{S\}}_5(c^{\{S\}}_5-c^{\{S\}}_6)(c^{\{S\}}_5-c^{\{S\}}_7)},\ \gamma_6 = \frac{1}{c^{\{S\}}_6(c^{\{S\}}_6-c^{\{S\}}_5)(c^{\{S\}}_6-c^{\{S\}}_7)},\\
    \gamma_7 &= \frac{1}{c^{\{S\}}_7(c^{\{S\}}_7-c^{\{S\}}_5)(c^{\{S\}}_7-c^{\{S\}}_6)},\ \gamma_8=\frac{1}{c^{\{S\}}_8(c^{\{S\}}_8-c^{\{S\}}_9)(c^{\{S\}}_8-c^{\{S\}}_{10})} \\
    \gamma_9&=\frac{1}{c^{\{S\}}_9(c^{\{S\}}_9-c^{\{S\}}_8)(c^{\{S\}}_9-c^{\{S\}}_{10})},\ \gamma_{10}=\frac{1}{c^{\{S\}}_{10}(c^{\{S\}}_{10}-c^{\{S\}}_8)(c^{\{S\}}_{10}-c^{\{S\}}_{9})}
\end{split}
\end{equation}

This method satisfies all available coupling conditions (up through fourth-order) and its base method satisfies all order conditions up through fifth-order when
\begin{align}
    c^{\{S\}}_9&=\frac{12-15c^{\{S\}}_{10}-15c^{\{S\}}_8+20 c^{\{S\}}_{10} c^{\{S\}}_8}{15-20c^{\{S\}}_{10}-20c^{\{S\}}_8+30c^{\{S\}}_{10}c^{\{S\}}_8}.
\end{align}

\end{document}